\documentclass[12pt,a4paper]{amsart}

\usepackage[utf8]{inputenc}
\usepackage[T1]{fontenc}
\usepackage[english]{babel}

\usepackage{indentfirst}
\usepackage{amssymb}
\usepackage{amsfonts}
\usepackage{amsmath}
\usepackage[top=2.5cm, bottom=2.5cm, left=2.5cm, right=2.5cm]{geometry}
\usepackage{amsopn}
\usepackage[colorlinks]{hyperref}
\usepackage{float}
\usepackage{graphicx}
\usepackage{color}
\usepackage{xspace,colortbl}
\usepackage{siunitx}
\usepackage{bbm}
\usepackage{enumitem}
\usepackage[dvipsnames]{xcolor}

\usepackage{amsthm}
\newtheorem{theorem}{Theorem}[section]
\newtheorem{corollary}[theorem]{Corollary}
\newtheorem{lemma}[theorem]{Lemma}

\theoremstyle{definition}
\newtheorem{definition}[theorem]{Definition}
\newtheorem{remark}[theorem]{Remark}

\newcommand{\cupdot}{\mathbin{\mathaccent\cdot\cup}}

\DeclareMathOperator*{\argmin}{arg\,min}

\DeclareMathOperator{\Var}{Var}
\DeclareMathOperator{\indicator}{\mathbbm{1}}

\DeclareSymbolFont{bbsymbol}{U}{bbold}{m}{n}
\DeclareMathSymbol{\ind}{\mathbin}{bbsymbol}{'061}
\title[Estimation of conditional inequality curves and measures]{Estimation of conditional inequality curves and measures via estimating the conditional quantile function}
\author[A{.} Jokiel-Rokita]{Alicja Jokiel-Rokita}
\author[S{.} Pi\c{a}tek]{Sylwester Pi\c{a}tek}
\author[R{.} Topolnicki]{Rafa{\l} Topolnicki}

\thanks{The work of S.P. was supported by the project Minigrants for doctoral students of the Wrocław University of Science and Technology.}

\address[A.J.-R. and S.P.]{Faculty of Pure and Applied Mathematics\\ Wroc{\l}aw University
	of Science and Technology\\
	Wybrze\.ze Wyspia\'nskiego 27,
	50-370 Wroc{\l}aw, Poland
}

\address[S.P. and R.T.]{Institute of Mathematics\\
Polish Academy of Sciences\\
	\'{S}niadeckich 8,
	00-656 Warsaw, Poland
}

\email{alicja.jokiel-rokita@pwr.edu.pl}

\email{sylwester.piatek@pwr.edu.pl}

\email{rafal.topolnicki@impan.pl}

\begin{document}

\begin{abstract}
In the paper conditional inequality curves and measures are proposed which allow us to describe the inequality/concentration of the conditional distribution of the feature we are interested in 
with respect to certain continuous variables.
Moreover, for a~graphical illustration of the change in values of the proposed conditional indices, a~curve of conditional inequality measures is introduced.
To estimate the curves and measures, a~new method is proposed to estimate the conditional quantile function. 
This method uses quantile regression estimates for a~given set of quantile orders, 
followed by isotonic regression on the estimated regression coefficients to ensure that the estimated conditional quantile function is nondecreasing.
The consistency of the proposed estimators is proved while their finite sample performance is evaluated through simulation studies and compared with existing approaches. 
Finally, practical application of conditional curves and measures is demonstrated by determining
 estimated curves of conditional salary inequalities with respect to years of experience in different employee tenure groups, 
based on some real data. 
% The code used to prepare this article was shared in the GitHub repository.
The code used to prepare the simulation results presented in this paper is available in a~dedicated GitHub repository.
% The code required to reproduce the results presented in this article is available in the dedicated GitHub repository.
\end{abstract}

\maketitle

%%%%%%%%%%%%%%%%%%%%%%%%%%%%%%%%%%%%%%%%%%%%%%%%%%%%%%%%%%%%%%%%%%%%%%%%%%%%%%%%%%%%%%%%%%%%%%%%%%%%%%%%%%%%%%%%%%%%%%%%%%%%%%
\section{Introduction}\label{sec:intro}
%%%%%%%%%%%%%%%%%%%%%%%%%%%%%%%%%%%%%%%%%%%%%%%%%%%%%%%%%%%%%%%%%%%%%%%%%%%%%%%%%%%%%%%%%%%%%%%%%%%%%%%%%%%%%%%%%%%%%%%%%%%%%%

In many practical applications, what is interesting is not only the value of the inequality index in the entire studied population but also in subpopulations, e.g., the value of the income inequality index in subpopulations based on gender or occupational group.
For example, Gini indices within some subpopulations, e.g., the population divided by sex (see, for example, \cite{Costa2023} or \cite{Larraz2014}) 
or by regions \cite{Jedrzejczak2014} or by various others \cite{Firpo2009} were considered and can be treated as a~type of conditional Gini coefficients.
However, these conditional measures are useful for examining the relationship between, for example, income inequality and some nominal covariates.
To the best of our knowledge, models of changes in the index value depending on certain continuous explanatory variables and a~graphical method of illustrating such changes have not been considered.
In this paper, we propose using quantile regression to model conditional inequality curves and indices and then use them to define a~curve of conditional inequality indices.
Such curves can be used to illustrate, for example, changes in income inequality indices depending on years of experience or, 
in medicine, to illustrate, for example, changes in the inequality indices of a~fixed blood parameter depending on the value of another continuous blood parameter or/and on age.

The most commonly used measure of inequality, applied, for example, in the analysis of income inequality, is the Gini coefficient \cite{Gini1912}, based on the Lorenz curve \cite{Lorenz1905}.
Since the Lorenz curve and the Gini index have some limitations (they are well defined only if the expected value of the observed random variable is finite), 
some alternative curves and indices have been proposed in recent years (see, for example, \cite{Prendergast2016} or \cite{Brazauskas2024}).
In this paper, we focus on the extension of the two quantile inequality curves $qZ$ and $qD$, which were proposed by Prendergast and Staudte \cite{Prendergast2018},
as well as the inequality measures associated with them, that is, $qZI$ and $qDI$, 
to the case of conditional inequality curves and measures.
These inequality curves and measures are defined below.

Let $Y$ be a~positive random variable with cumulative distribution function (cdf) $F(t)=P(Y\leq t).$
Denote
\begin{equation*}
Q(p)=F^{-1}(p)=\inf\{t\in\mathbb{R}:F(t)\geq p\},
\end{equation*}
the quantile function (qf) of $Y,$ where $p\in(0,1).$
The inequality curves $qZ$ and $qD$, corresponding to the random variable $Y$, are defined as the sets $\{(p,qZ(p;Q)):p\in [0,1]\}$, $\{(p,qD(p;Q)):p\in [0,1]\}$, respectively, where
\begin{equation}\label{e:qZ}
qZ(p;Q)=1-\frac{Q\left(\tfrac{p}{2}\right)}{Q\left(\tfrac{1+p}{2}\right)},
\end{equation}
\begin{equation}\label{e:qD}
qD(p;Q)=1-\frac{Q\left(\tfrac{p}{2}\right)}{Q\left(1-\tfrac{p}{2}\right)},
\end{equation}
for $p\in(0,1),$ and $qZ(0;Q)=qZ(1;Q)=qD(0;Q)=1,$ $qD(1;Q)=0.$
Quantile-based counterparts to the Gini index based on $qZ$ and $qD$ are the indices $qZI$ and $qDI$ defined as the area under $qZ$ and $qD$ curves respectively, i.e.,
\begin{equation}\label{e:qZI}
qZI(Q)=\int_{0}^{1}qZ(p;Q)\,dp,
\end{equation}
\begin{equation}\label{e:qDI}
qDI(Q)=\int_{0}^{1}qD(p;Q)\,dp.
\end{equation}

\begin{figure}
    \centering
    \includegraphics[width=1.0\linewidth]{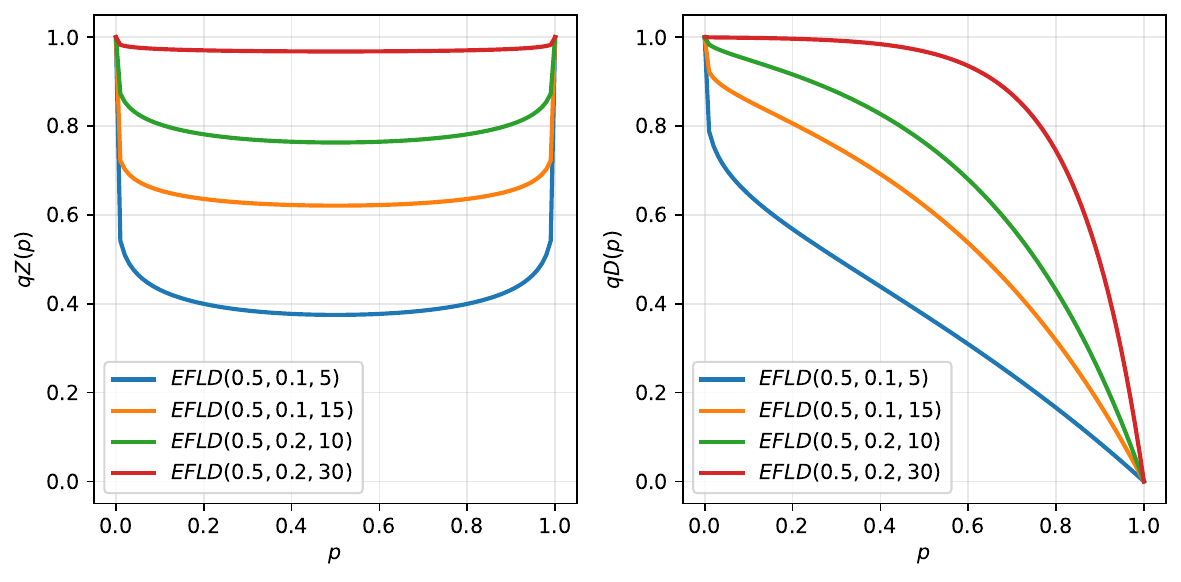}
    \caption{Example of the $qZ$ and $qD$ curves for random variables from the exponent of FLD distribution (defined in Section \ref{sec:simulation_study}) with various sets of parameters}
    \label{fig:qZ_qD_EFLD_example}
\end{figure}

Figure \ref{fig:qZ_qD_EFLD_example} shows examples of the $qZ$ and $qD$ curves corresponding random variables from the exponent of FLD distribution (defined in Section \ref{sec:simulation_study}) with various parameters.
The larger the income inequality, the closer they are to the level of 1.

Suppose now that the random variable $Y$ is associated with a~random vector $X$ (possibly multivariate) of some covariates,
and we are interested in the inequality curves and measures corresponding to the conditional distribution $Y|X=x,$ for $x\in \mathcal{X},$ where $\mathcal{X}$ is the value space of $X.$
Let $Q_{x}$ be the quantile function of  $Y|X=x.$
Conditional equivalents of inequality curves $qZ(\cdot;Q)$, $qD(\cdot;Q)$ and measures $qZI(Q)$, $qDI(Q)$, given $x$, are $qZ(\cdot;Q_{x})$, $qD(\cdot;Q_{x})$ and $qZI(Q_{x})$, $qDI(Q_{x})$, respectively.
% When $x\in{\mathbb R}$ or $Q_x$ depends on the value of a~real function of a~vector $x$ of covariates, 
% we define the following curves of conditional inequality measures 
% \begin{equation}\label{e:CqZI_CqDI}
% \{(x,qZI(Q_{x}))\colon x\in \mathcal{X}\}  \ \ {\rm and} \ \ \{(x,qDI(Q_{x}))\colon x\in \mathcal{X}\}.
% \end{equation}
\begin{definition}
For the inequality measures $qZI$ and $qDI$
we define the \textbf{graphs of conditional inequality measures}
\begin{equation}\label{e:CqZI_CqDI}
\{(x,qZI(Q_{x}))\colon x\in \mathcal{X}\}  \ \ {\rm and} \ \ \{(x,qDI(Q_{x}))\colon x\in \mathcal{X}\}.
\end{equation}
In the case when $X$ is one-dimensional, these objects reduce to the \textbf{curves of conditional inequality measures}. 
These curves of conditional inequality measures can also be obtained by fixing some values on all coordinates of $X$ except one.
\end{definition}

To the best of our knowledge, such curves have not been considered in the literature.
Examples of estimated curves \eqref{e:CqZI_CqDI} are included in Figure \ref{fig:four_states_qZI_qDI} and show their usefulness to illustrate changes in conditional inequality indices with respect to the covariate.

The paper concerns the estimation of the conditional inequality curves \ $qZ(\cdot;Q_{x})$, \ $qD(\cdot;Q_{x})$ and the measures $qZI(Q_{x})$, $qDI(Q_{x})$, 
as well as the curve of conditional inequality indices $\{(x,qZI(Q_{x})\colon x\in \mathcal{X}\}$ and  $\{(x,qDI(Q_{x})\colon x\in \mathcal{X}\}$.
Since these curves and measures depend on the conditional quantile functions $Q_{x}$, $x\in \mathcal{X}$, in Section~\ref{sec:measures} we propose several methods for estimating these functions.
Section \ref{sec:properties} contains theorems concerning the asym\-ptotic properties of estimators proposed in Section~\ref{sec:measures}.
The proofs of these theorems are included in the Appendix.
Section~\ref{sec:simulation_study} contains a~description and results of the simulation study, the aim of which was to compare the proposed estimators with those known from the literature.
In Section~\ref{sec:real_data} an example of real data analysis is performed with the methods proposed in Section~\ref{sec:measures}.
The last Section~\ref{sec:conclusion} concludes the results of this paper and contains some prospects.

%%%%%%%%%%%%%%%%%%%%%%%%%%%%%%%%%%%%%%%%%%%%%%%%%%%%%%%%%%%%%%%%%%%%%%%%%%%%%%%%%%%%%%%%%%%%%%%%%%%%%%%%%%%%%%%%%%%%%%%%%%%%%%%%%%%%%%%%%%%%%%%%%%%%%%%%%%%%%%%%%%%%%%%%%%%%%%%%%%%%%%%%%%%%%%%%%%%%%%%%%%%%
\section{Estimation of conditional inequality curves and inequality measures}\label{sec:measures}
%%%%%%%%%%%%%%%%%%%%%%%%%%%%%%%%%%%%%%%%%%%%%%%%%%%%%%%%%%%%%%%%%%%%%%%%%%%%%%%%%%%%%%%%%%%%%%%%%%%%%%%%%%%%%%%%%%%%%%%%%%%%%%%%%%%%%%%%%%%%%%%%%%%%%%%%%%%%%%%%%%%%%%%%%%%%%%%%%%%%%%%%%%%%%%%%%%%%%%%%%%%%%%%%%

Let $(X_{1},\ldots ,X_{d}, Y)$ be a~random vector such that for each $x\in \mathcal{X},$ 
where $\mathcal{X}\subset {\mathbb R}_{+}^{d}:=[0,\infty)^d$ is the value space of the random variable $X=(X_{1},\ldots,X_{d}),$
the cumulative distribution function $F_{x}$ of the conditional distribution of $Y|X=x$ is continuous and increasing on $(0,\infty),$ $F_{x}(y)=0$ for $y\leq 0.$
Since the conditional inequality curves $qZ(\cdot;Q_{x})$ and $qD(\cdot;Q_{x})$, given $x$, are expressed by the quantile function $Q_{x}$, 
the natural approach to modelling and estimating these conditional curves is to use quantile regression.
However, note that the quantile regression model concerns the quantile of a~given order, say $p$, as a~function of $x$, 
while the conditional inequality curves considered are graphs of functions of the argument~$p$, for a~given value of $x$.
Let $Z=\psi(Y)$, where $\psi\colon[0,\infty)\rightarrow \mathbb{R}$ is a~known continuous and strictly increasing function, and assume that the quantile function of $Z|X=x$ is of the form
\begin{equation}\label{eq:assumption_ln}
\beta_{0}(p)+\beta_{1}(p)x_{1}+\cdots +\beta_{d}(p)x_{d},
\end{equation}
where the functions $\beta_{i},$ $i=0,1,\ldots,d,$ are nondecreasing and continuous on $(0,1)$.
The assumption of monotonicity of the functions $\beta_{i}$, $i=0,\ldots,d,$ 
ensures the appropriate monotonicity of the conditional quantile functions $Q_{x}$ for each $x\in\mathcal{X}$ in the model considered, in particular when the set $\mathcal{X}$ is unbounded.
Under the above assumptions, the quantile function $Q_{x}$ of $Y|X=x$ is of the form
\begin{equation}\label{eq:assumption}
Q_{x}(p)=\psi^{-1}\left(\beta_{0}(p)+\beta_{1}(p)x_{1}+\cdots +\beta_{d}(p)x_{d}\right)
\end{equation}
and
% \begin{equation}\label{e:qZx}
% qZ(p;Q_{x})=1-\frac{\psi^{-1}\left(\beta_{0}(p/2)+\beta_{1}(p/2)x_{1}+\cdots +\beta_{d}(p/2)x_{d}\right)}{\psi^{-1}\left(\beta_{0}(1/2+p/2)+\beta_{1}(1/2+p/2)x_{1}+\cdots +\beta_{d}(1/2+p/2)x_{d}\right)},
% \end{equation}
% \begin{equation}\label{e:qDx}
% qD(p;Q_{x})=1-\frac{\psi^{-1}\left(\beta_{0}(p/2)+\beta_{1}(p/2)x_{1}+\cdots +\beta_{d}(p/2)x_{d}\right)}{\psi^{-1}\left(\beta_{0}(1-p/2)+\beta_{1}(1-p/2)x_{1}+\cdots +\beta_{d}(1-p/2)x_{d}\right)},
% \end{equation}
\begin{equation}\label{e:qZx}
qZ(p;Q_{x})=1-\frac{\psi^{-1}\left(\beta_{0}\left(\tfrac{p}{2}\right)+\beta_{1}\left(\tfrac{p}{2}\right)x_{1}+\cdots +\beta_{d}\left(\tfrac{p}{2}\right)x_{d}\right)}{\psi^{-1}\left(\beta_{0}\left(\tfrac{1+p}{2}\right)+\beta_{1}\left(\tfrac{1+p}{2}\right)x_{1}+\cdots +\beta_{d}\left(\tfrac{1+p}{2}\right)x_{d}\right)},
\end{equation}
\begin{equation}\label{e:qDx}
qD(p;Q_{x})=1-\frac{\psi^{-1}\left(\beta_{0}\left(\tfrac{p}{2}\right)+\beta_{1}\left(\tfrac{p}{2}\right)x_{1}+\cdots +\beta_{d}\left(\tfrac{p}{2}\right)x_{d}\right)}{\psi^{-1}\left(\beta_{0}\left(1-\tfrac{p}{2}\right)+\beta_{1}\left(1-\tfrac{p}{2}\right)x_{1}+\cdots +\beta_{d}\left(1-\tfrac{p}{2}\right)x_{d}\right)}.
\end{equation}
for $p\in(0,1),$ and $qZ(0;Q_{x})=qZ(1;Q_{x})=qD(0;Q_{x})=1,$ $qD(1;Q_{x})=0.$
The conditional quantile inequality curves and measures can be estimated by plugging an estimator of $Q_{x}$ into formulae \eqref{e:qZ}, \eqref{e:qD} and \eqref{e:qZI}, \eqref{e:qDI}, respectively.
In our model, equivalently the estimators $\hat{\beta}_{i}$ of $\beta_{i},$ $i=0,1,\ldots,d,$ can be plugged into the formulae \eqref{e:qZx} and \eqref{e:qDx} 
to obtain the plug-in estimators of $qZ(\cdot;Q_{x})$ and $qD(\cdot;Q_{x})$, respectively.
Then, the indices $qZI(Q_{x})$ and $qDI(Q_{x})$ can be estimated by integrating (numerically) these plug-in estimators of $qZ(\cdot;Q_{x})$ and $qD(\cdot;Q_{x}).$

Let $((x_{11},\ldots,x_{1d}),y_1), \ldots, ((x_{n1},\ldots,x_{nd}),y_n)$
be realisations of independent random variables from the distribution of $(X,Y)$ and let $z_{i}=\psi(y_{i}),$ $i=1,\ldots,n.$
To estimate the functions $\beta_{i},$ $i=0,1,\ldots,d,$ we propose to use methods of estimation of quantile regression function \eqref{eq:assumption_ln}
on the basis of $((x_{11},\ldots,x_{1d}),z_1), \ldots ,((x_{n1},\ldots,x_{nd}),z_n)$ (for more details concerning quantile regression, see \cite{Koenker1978}, \cite{Koenker2005}).
 However, in the estimation of the quantile regression function, 
we are usually interested in estimating the values of the functions $\beta_{i}$ for several values of~$p$, e.g., for $p\in\{1/10, 1/4, 1/2, 3/4, 9/10\}.$
In the problem considered, we have to estimate entire functions $\beta_{i},$ for each $i=0,1,\ldots,d.$
We propose several methods of estimating the function $\beta_{i}$ based on various approaches to the estimation of coefficients $\beta_{i}(p_{j})$ of the quantile regression, 
for $i=0,1,\ldots,d$ and fixed points $0<p_{1}<p_{2}<\ldots<p_{m-1}<1$.
In each method, the estimators $\hat{\beta}_i$ of the functions $\beta_i$, where $i=0,1,\ldots,d$, 
on the interval $(p_{j},p_{j+1}), j\in\{1,\ldots,m-2\}$ are obtained by a~linear interpolation between the estimated values 
$\hat{\beta}_i(p_j)$ and $\hat{\beta}_i(p_{j+1})$ for each $p_j\in\{p_1,p_2,\ldots,p_{m-1}\}$.
This means that for each $p$ from the interval $(p_j,p_{j+1})$, for each $j\in\{1,\ldots,m-2\}$, we have
\begin{align}\label{eq:lin_interpol}
    \hat{\beta}_i(p) = \hat{\beta}_i(p_j) + (p-p_j)
    \frac{\hat{\beta}_i(p_{j+1})-\hat{\beta}_i(p_j)}{p_{j+1}-p_j}.
\end{align}

The method described above leads to the estimation of the conditional quantile function on an interval $(p_{1},p_{m-1})$. 
Since here we focus on the estimation of the conditional inequality curves and measures, not the conditional quantile function, 
we do not suggest any method of  estimating this function on intervals $(0,p_{1})$ and $(p_{m-1},1)$.
The inequality curves $qZ(\cdot;Q_{x})$ and $qD(\cdot;Q_{x})$ for $p\in(0,p_{1})$ will be estimated by linear interpolation between points $(0,1)$, $(p_{1},qZ(p_{1};\hat{Q}_{x}))$ and between points $(0,1)$, $(p_{1},qD(p_{1};\hat{Q}_{x}))$, respectively, and similarly the inequality curves $qZ(\cdot;Q_{x})$ and $qD(\cdot;Q_{x})$ will be estimated on an interval $(p_{m-1},1)$.

In the following, we present methods for estimating the coefficients $\beta_{i}(p)$, $i=0,1,\ldots,d,$ of the model \eqref{eq:assumption_ln}, for a~given $p$,
as an element of estimation of functions $\beta_{i}$ and, consequently, estimation of conditional inequality curves and measures.

Considering an extended vector $x_k$ with $x_{k0}=1$, we have $\beta^{\top}x_k = \beta_0 + \beta_1x_{k1} + \ldots  + \beta_dx_{kd}$ for $k=1,\ldots,n$.
The most common estimator of the quantile regression coefficients $\beta_{i}(p),$ $i=0,1,\ldots,d,$ 
is the estimator $\hat{\beta}^{OQR}(p)=(\hat{\beta}_{0}^{OQR}(p),\hat{\beta}_{1}^{OQR}(p), \ldots,\hat{\beta}_{d}^{OQR}(p))$ given by
\begin{equation}\label{eq:oqr_beta}
     \hat{\beta}^{OQR}(p)=\argmin_{\beta} \sum_{k=1}^n \rho_p(z_k-\beta^{\top}x_k),
\end{equation}
where
\begin{align}\label{eq:loss_function}
\rho_p(u) = u\left[p - \indicator(u < 0)\right].
\end{align}
Thus, equivalently to \eqref{eq:oqr_beta}, up to a~positive constant, we have
\begin{equation}\label{eq:oqr_beta_al}
      \hat{\beta}^{OQR}(p) = \argmin_{\beta} \sum_{k=1}^n
    \left[|z_{k}-\beta^{\top}x_k| + (2p-1)(z_{k}-\beta^{\top}x_k)\right].
\end{equation}

An estimator of $\beta(p)=(\beta_{0}(p),\beta_{1}(p),\ldots,\beta_{d}(p))$ can also be obtained by replacing the absolute value function in \eqref{eq:oqr_beta_al} with its smooth approximation.
% The advantage of this approach is that the minimised function in formula \eqref{eq:oqr_beta_al} is smooth, which may allow faster optimisation procedures.
The advantage of this approach is that after this replacement, the minimised function given in formula \eqref{eq:oqr_beta_al} will become smooth, which may allow faster optimisation procedures.
Several approaches have been proposed to smooth this function (see, for example, \cite{Horowitz1998}, \cite{Muggeo2012}, \cite{Saleh2021}, \cite{He2023}).
In order to approximate the absolute value in \eqref{eq:oqr_beta_al}, we will (partially) use Hutson's \cite{Hutson2024} proposal %to replace the absolute value function with %the approximating function $g_{\tau}$ given by
\begin{equation}\label{eq:g_tau}
g_{\tau}(u)=
\tau\left[\log\left(e^{-\frac{u}{\tau}}+1\right)
+\log\left(e^{\frac{u}{\tau}}+1\right)\right],
\end{equation}
and the function 
\begin{equation*}
h_\tau(u)=u\tanh\left(\frac{u}{\tau}\right),
\end{equation*}
where $\tau$ is a~smoothing parameter. 
The function~$g_{\tau}$ can significantly overestimate the function $|\cdot|$ in the neighbourhood of 0.
On the other hand, the function $h_\tau$ underestimates the absolute value function in the neighbourhood of 0 (see \cite{Bagul2021}).
For this reason, our proposal is to use the function
\begin{equation}\label{e:function_f_tau}
f_{\tau}(u)=\frac{g_{\tau}(u)+h_{\tau}(u)}{2}.
\end{equation}
%with one of these four functions to receive a~better approximation of $|\cdot|$.
%We chose $h_{\tau}(u)$ because of its simplicity and denoted $f_{\tau}(u)=(g_{\tau}(u)+h_{\tau}(u))/2$, expecting to receive an estimator of $\beta$ with a~smaller bias or MSE than when $g_{\tau}$ or $h_{\tau}$ is used.
Taking into account formula \eqref{eq:oqr_beta_al} and approximating function \eqref{e:function_f_tau},  we define the estimator 
\begin{equation}\label{eq:aqr}
    \hat{\beta}_{\tau}^{AQR}(p) = \argmin_{\beta} \sum_{k=1}^n
    \left[f_{\tau}(z_{k}-\beta^{\top}x_k) + (2p-1)(z_{k}-\beta^{\top}x_k)\right].
\end{equation}
When $\tau$ goes to zero, the estimator $\hat{\beta}_{\tau}^{AQR}$ tends to $\hat{\beta}^{OQR}$, since the approximation $f_{\tau}$ converges uniformly to $|\cdot|$ when $\tau$ tends to 0.
In Section \ref{sec:simulation_study} we provide a~proposal for selecting the parameter $\tau.$

Let $0<p_{1}<p_{2}<\ldots <p_{m-1}<1$ be the fixed orders of quantiles, 
and $\hat{\beta}^{OQR}(p_{j})$, $\hat{\beta}_{\tau}^{AQR}(p_{j}),$ $j=1,\ldots,m-1,$ be the quantile regression coefficient estimators described above.
We have no guarantee that $\hat{\beta}_{i}^{OQR}(p_{1})\leq \hat{\beta}_{i}^{OQR}(p_{2})\leq \ldots\leq \hat{\beta}_{i}^{OQR}(p_{m-1}),$ 
and $\hat{\beta}_{\tau, i}^{AQR}(p_{1})\leq \hat{\beta}_{\tau, i}^{AQR}(p_{2})\leq \ldots \leq \hat{\beta}_{\tau, i}^{AQR}(p_{m-1})$, for each $i=0, 1,\ldots d.$
Therefore, the construction of the functions $\beta_{i}$ estimates, based on the linear interpolation described by formula \eqref{eq:lin_interpol},
may lead to functions that are not nondecreasing.
As a~consequence, the estimators of conditional quantile functions $Q_{x}(p)=\beta_{0}(p)+\beta_{1}(p)x_{1}+\ldots+\beta_{d}(p)x_{d},$ $x\in\mathcal{X},$
may not be nondecreasing.
To enforce the monotonicity of the estimates $\hat{\beta}_{i}^{OQR}$ and $\hat{\beta}_{\tau, i}^{AQR}$ of the functions $\beta_{i}$ for each $i\in\{0,1,\ldots,d\},$ 
we propose using, for each $i$, the isotonic regression (see, for example, \cite{Barlow1972}) to $\hat{\beta}_i^{OQR}(p_{1}),\ldots,\hat{\beta}_i^{OQR}(p_{m-1})$ and 
$\hat{\beta}_{\tau, i}^{AQR}(p_{1}),\ldots,\hat{\beta}_{\tau, i}^{AQR}(p_{m-1})$, respectively, treated as data.
Therefore, $\tilde{\beta}_i^{OQR}(p_j)$ (and analogously $\tilde{\beta}_{\tau, i}^{AQR}(p_j)$) are the values minimising
\begin{align}
    \sum_{j=1}^{m-1} w_j\left[\tilde{\beta}_i^{OQR}(p_j)-\hat{\beta}_i^{OQR}(p_j)\right]^2,
\end{align}
where $w_j$, for $j\in\{1,\ldots,m-1\}$, are positive weights, 
with the condition that $\tilde{\beta}_i^{OQR}(p_j) \leq \tilde{\beta}_i^{OQR}(p_{j+1})$ if $p_j \leq p_{j+1}$ for any $i\in\{0,1,\ldots,d\}$ (see, for example, \cite{Best1990}). 
Note that the $L^2$ minimiser over nondecreasing sequences is unique, so the isotonic fit is well defined.
In the simulation study, described in Section \ref{sec:simulation_study}, we assumed $w_j=1$ for each $j$.
This approach leads to nondecreasing sequences $\tilde\beta_{i}^{OQR}(p_{1}),\ldots,\tilde\beta_{i}^{OQR}(p_{m-1})$ 
and $\tilde\beta_{\tau, i}^{AQR}(p_{1}),\ldots,\tilde\beta_{\tau, i}^{AQR}(p_{m-1})$.
The values of $\tilde{\beta}_i^{OQR}(p)$ and $\tilde{\beta}_{\tau, i}^{AQR}(p)$ for any $p$ are obtained by linear interpolation according to formula \eqref{eq:lin_interpol}.
An example of estimates $\hat{\beta}_{0}^{OQR}$, $\tilde{\beta}_{0}^{OQR}$ and  $\hat{\beta}_{1}^{OQR}$, $\tilde{\beta}_{1}^{OQR}$ 
in the case of $d=1,$ obtained on the basis of simulated data is presented in Figure \ref{fig:iso_regr_betas_fld}.

The approach described above leads to estimators 
\begin{equation}\label{e:Q_IOQR_1}
\tilde{Q}_{x}^{OQR}(p)=\tilde{\beta}_{0}^{OQR}(p)+\tilde{\beta}_{1}^{OQR}(p)x_{1}+\ldots+\tilde{\beta}_{d}^{OQR}(p)x_{d},
\end{equation}
\begin{equation}\label{e:Q_IOQR_2}
\tilde{Q}_{x}^{AQR}(p)=\tilde{\beta}_{\tau,0}^{AQR}(p)
+\tilde{\beta}_{\tau,1}^{AQR}(p)x_{1}+\ldots+\tilde{\beta}_{\tau,d}^{AQR}(p)x_{d}
\end{equation}
of the conditional quantile functions $Q_{x}$, $x\in\mathcal{X}$, whose realisations are nondecreasing functions, 
and the plug-in estimators of the conditional inequality curves $qZ(\cdot;Q_{x})$, $qD(\cdot;Q_{x})$ and measures $qZI(Q_{x})$, $qDI(Q_{x}).$

\begin{remark}
    Isotonic regression is sometimes used together with quantile regression, but to impose monotonicity of $Q_x(p)$ with respect to $x$ for a~fixed $p$ (see, for example, \cite{Abrevaya2005} or \cite{Moesching2020}).
    In this paper, isotonic regression is applied to obtain the monotonicity of $Q_x(p)$ with respect to $p$ for each fixed $x$.
\end{remark}

\begin{figure}
    \centering
    \includegraphics[width=\textwidth]{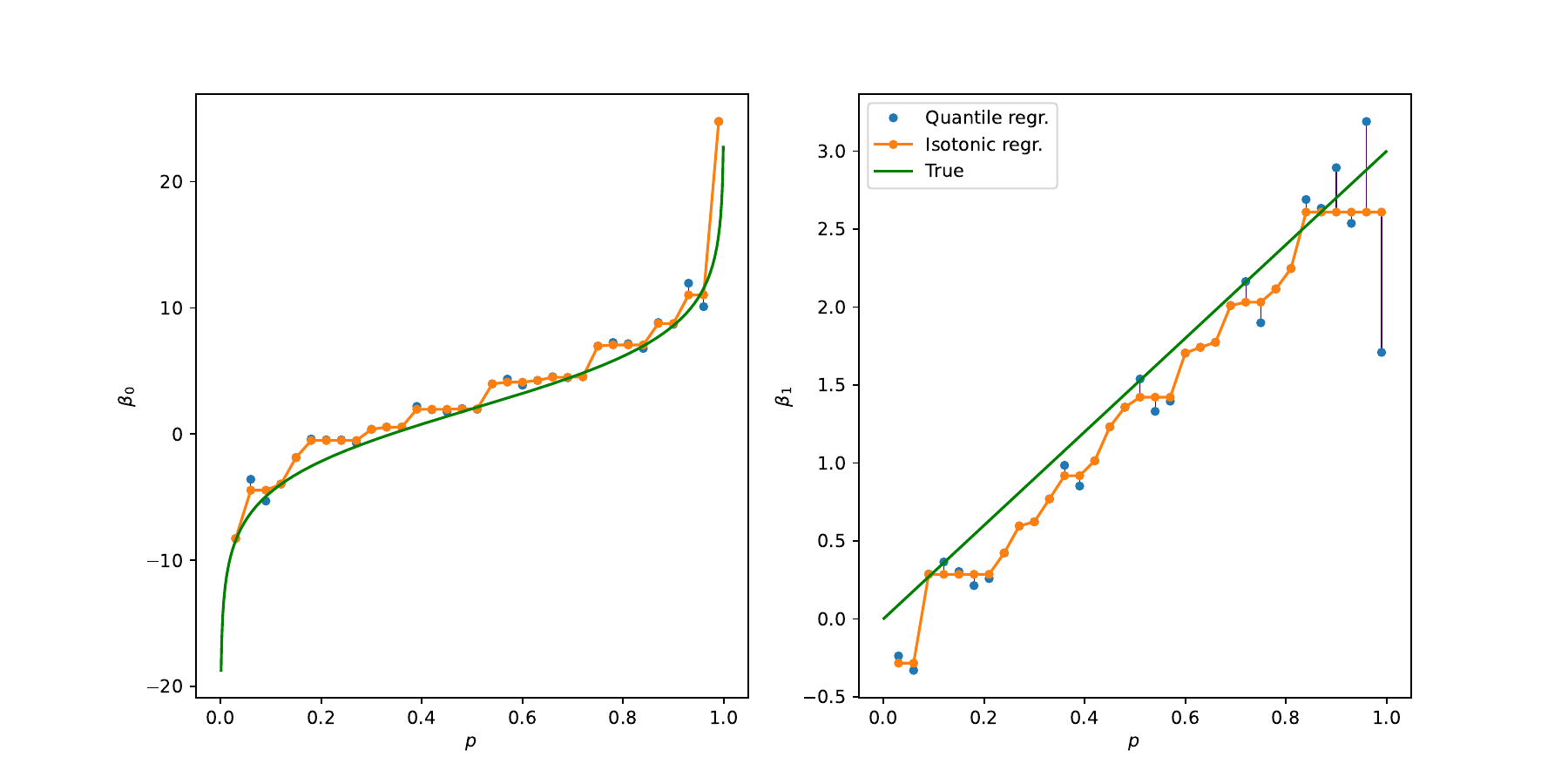}
    \caption{Isotonic regression on $\hat{\beta}_0(p)$ and $\hat{\beta}_1(p)$ compared to the true value of ${\beta}_0(p)$ (left) and ${\beta}_1(p)$ (right) for a~sample of data generated from flattened logistic distribution $FLD(2, 3, \kappa)$ with $\kappa\in[0,10]$ (see Section \ref{sec:fld}). Estimation based on 33 equidistant orders of quantiles.}
    \label{fig:iso_regr_betas_fld}
\end{figure}

%%%%%%%%%%%%%%%%%%%%%%%%%%%%%%%%%%%%%%%%%%%%%%%%%%%%%%%%%%%%%%%%%%%%%%%%%%%%%%%%%%%%%%%%%%%%%%%%%%%%%
%%%%%%%%%%%%%%%%%%%%%%%%%%%%%%%%%%%%%%%%%%%%%%%%%%%%%%%%%%%%%%%%%%%%%%%%%%%%%%%%%%%%%%%%%%%%%%%%%%%%%
\section{Properties of the estimators of conditional quantile measures}\label{sec:properties}
%%%%%%%%%%%%%%%%%%%%%%%%%%%%%%%%%%%%%%%%%%%%%%%%%%%%%%%%%%%%%%%%%%%%%%%%%%%%%%%%%%%%%%%%%%%%%%%%%%%%%%%%%%%%%%%%%%%%%%%%%%%%%%%%%%%%%%%%%%%%%%%%%%%%%%%%%%%%%%%%%%%%%%%%%%%%%%%%%%%%%%%%%%%%%%%%%%%%%%%

This section contains the theorem regarding the asymptotic properties of the estimators of the function 
$\beta$, the conditional quantile function $Q_{x}$, and the plug-in estimators of the conditional inequality curves and measures defined in Section \ref{sec:measures}.
The proof of this theorem is included in Appendix.
 
We assume the following additional conditions to ensure the consistency of estimators $\hat{\beta}^{OQR}(p)$, for each $p\in(0,1)$ (see \cite{Angrist2006}).
\begin{enumerate}[label={(A}\arabic*{)}]
    \item $(Z_{k},X_{k})$, $k=1,\ldots,n,$ are iid on a~probability space $(\Omega,\mathcal{F},P)$ for each $n$;
    \item The conditional density $f_{Z}(z|X=x)$ exists $P$-a.s.;
    \item $\mathbb{E}||X|| <\infty$, and for all $p\in[a,b]\subset (0,1)$ there exists the unique solution $\beta(p)$ to the equation $\mathbb{E}\{[p-\indicator(Z\leq X^{\top}\beta(p))]X\}=0$.
    By $||X||$ we denote the Euclidean norm of the vector $X$.
\end{enumerate}

In Theorem \ref{thm:IOQR_conv}, we use the following notation for the proposed estimators.
That is, 
$\hat{\beta}^{(m,n)} = \left(\hat{\beta}_0^{(m,n)}, \dots, \hat{\beta}_d^{(m,n)}\right)$ denotes the estimator of $\beta$ 
based on $n$ copies of $(X,Z)$ such that $\hat{\beta}_i^{(m,n)}$ is defined as a~linear interpolation between points $\left(\frac{1}{m}, \hat{\beta}_i^{(m,n)}(\frac{1}{m})\right),$ 
$\ldots,\left(\frac{m-1}{m}, \hat{\beta}_i^{(m,n)}(\frac{m-1}{m})\right)$ for each $i \in \{0,\ldots,d\}$ and $\hat{\beta}_i^{(m,n)}(\frac{k}{m})$, $k=1,\ldots,m-1,$ obtained by the OQR method,
and $\tilde{\beta}^{(m,n)}$ is the estimator of $\beta$ obtained analogously using the IOQR method.
Moreover, $\hat{Q}_x^{(m,n)}$ and $\tilde{Q}_x^{(m,n)}$ denote the plug-in estimator of $Q_x$ based on $\hat{\beta}^{(m,n)}$ and $\tilde{\beta}^{(m,n)}$, respectively.

\begin{theorem}\label{thm:IOQR_conv}
Under assumptions (A1)--(A3), for $m,n\to\infty$ we have:
    \begin{enumerate}[label=(\roman*)]
        \item $\hat{\beta}^{(m,n)} \xrightarrow{P} \beta$ on $[\eta,1-\eta]$ for each $\eta\in(0,\frac{1}{2})$;
        \item $\tilde{\beta}^{(m,n)} \xrightarrow{P} \beta$ on $[\eta,1-\eta]$ for each $\eta\in(0,\frac{1}{2})$;
        \item for  $x\in{\mathcal X},$ $\hat{Q}_{x}^{(m,n)} \xrightarrow{P} Q_{x}$ and
        $\tilde{Q}_{x}^{(m,n)} \xrightarrow{P} Q_{x}$ on $[\eta,1-\eta]$ for each $\eta\in(0,\frac{1}{2})$;
        \item for  $x\in{\mathcal X},$ ${qZ}\left(\cdot;\hat{Q}_{x}^{(m,n)}\right) \xrightarrow{P} qZ(\cdot;Q_{x}),$ and \\
        ${qZ}\left(\cdot;\tilde{Q}_{x}^{(m,n)}\right) \xrightarrow{P} qZ(\cdot;Q_{x})$ on $[2\eta,1-2\eta]$ for each $\eta\in(0,\frac{1}{4})$; 
        \item for  $x\in{\mathcal X},$ $\int_{2\eta}^{1-2\eta}{qZ}\left(p;\hat{Q}_{x}^{(m,n)}\right) dp\xrightarrow{P} \int_{2\eta}^{1-2\eta}qZ(p;Q_{x})dp,$ and \\
              $\int_{2\eta}^{1-2\eta}{qZ}\left(p;\tilde{Q}_{x}^{(m,n)}\right) dp\xrightarrow{P} \int_{2\eta}^{1-2\eta}qZ(p;Q_{x})dp$.
     \end{enumerate}
\end{theorem}

\begin{remark}
 The points (iv) and (v) of Theorem \ref{thm:IOQR_conv} also hold for $qD$ and its estimator.
\end{remark}

%%%%%%%%%%%%%%%%%%%%%%%%%%%%%%%%%%%%%%%%%%%%%%%%%%%%%%%%%%%%%%%%%%%%%%%%%%%%%%%%%%%%%%%%%%%%%%%%%%%%%
%%%%%%%%%%%%%%%%%%%%%%%%%%%%%%%%%%%%%%%%%%%%%%%%%%%%%%%%%%%%%%%%%%%%%%%%%%%%%%%%%%%%%%%%%%%%%%%%%%%%%
\section{Simulation study}\label{sec:simulation_study}
%%%%%%%%%%%%%%%%%%%%%%%%%%%%%%%%%%%%%%%%%%%%%%%%%%%%%%%%%%%%%%%%%%%%%%%%%%%%%%%%%%%%%%%%%%%%%%%%%%%%%
%%%%%%%%%%%%%%%%%%%%%%%%%%%%%%%%%%%%%%%%%%%%%%%%%%%%%%%%%%%%%%%%%%%%%%%%%%%%%%%%%%%%%%%%%%%%%%%%%%%%%

A~simulation study was conducted to compare the performance of the plug-in estimators of conditional $qZI$ and $qDI$.
The samples in this experiment are drawn from the flattened logistic distribution (FLD), described in Section~\ref{sec:fld}, and then exponentiated to obtain positive values (we assumed $\psi(\cdot)=\log(\cdot)$).
This experiment was conducted using Python. More technical details are given in Section \ref{sec:technical_details}.

\subsection{Flattened logistic distribution (FLD)}\label{sec:fld}
Sharma and Chakrabarty introduced FLD (see formula (1.3) in \cite{Sharma2019}). Its qf is given by the formula
\begin{equation}\label{e:q_fld}
    Q(p)=\alpha + \beta\left( \log\frac{p}{1-p} + \kappa p \right),
\end{equation}
where $\alpha\in\mathbb{R}$ is the location parameter, $\beta>0$ is the scale parameter and $\kappa>0$ is the shape parameter that regulates the flatness of the peak of the density function of this distribution.

We will denote the distribution with quantile function \eqref{e:q_fld} by $FLD(\alpha,\beta,\kappa)$.
For a~random variable $Z$ from $FLD$, we denote $EFLD$ as a~distribution of $\exp(Z)$.
In the simulation study, we assume that $Z|X=x$, i.e., $\log Y|X=x$, 
has the following quantile function 
\begin{equation*}
    Q_x(p):=Q(p|X=x) = \beta_0(p) + \beta_1(p)x,
\end{equation*}
where
\begin{equation*}
    \beta_0(p) = \alpha+\beta\log\frac{p}{1-p}
    \qquad
    \text{and}
    \qquad
    \beta_1(p) = \beta \gamma p,
\end{equation*}
i.e., $Z|X=x \sim FLD(\alpha,\beta,\gamma x),$ and 
$Y|X=x \sim EFLD(\alpha,\beta,\gamma x)$.
The scatterplot of the example of data generated from $FLD(\alpha,\beta,\gamma x)$ is presented in Figure \ref{fig:scatter_plot_fld}, while the scatterplot of its exponent ($EFLD(\alpha,\beta,\gamma x)$) is depicted in Figure \ref{fig:scatter_plot_exp_fld}. It can be seen that the variance of FLD increases with increasing $x$.

It can be shown (see Lemma \ref{lem:moments_efld} in Section \ref{sec:momentsEFLD}) that the mean of EFLD is not finite when $\beta\geq1$ and the variance is not finite when $\beta\geq0.5$.
Thus, for $\beta\geq1$, it is not recommended to use the Gini index or other measures defined with the mean value, since they are undefined.

% For such a~r.v. $Y$, some inequality curves and measures are given with a~closed-form expression. For instance, $qZ$ and $qD$ are given by

For the $EFLD(\alpha,\beta,\kappa)$ with $\kappa=\gamma x$, the $qZ$ and $qD$ are given by
\begin{align*}
& qZ(p;Q_x)=1-\left[\frac{p(1-p)}{(1+p)(2-p)}\right]^\beta \exp \left(- \beta \gamma x / 2\right), \\
& qD(p;Q_x)=1-\left[\frac{p}{2-p}\right]^{2 \beta} \exp \left[\beta \gamma x(p-1)\right].
\end{align*}

%%%%%%%%%%%%%%%%%%%%%%%%%%%%%%%%%%%%%%%%%%%%%%%
%%%%%% TWO FIGURES next to each other %%%%%%%%% %%%%%%%%%%%%%%%
\par\noindent
%%%%%%%%%%%%%%%%%%%%%%%%%%%%%%%%%%%%%%%%%%%%%
\begin{minipage}[b]{0.45\textwidth}
\begin{figure}[H]
    \centering
    \includegraphics[width=\textwidth]{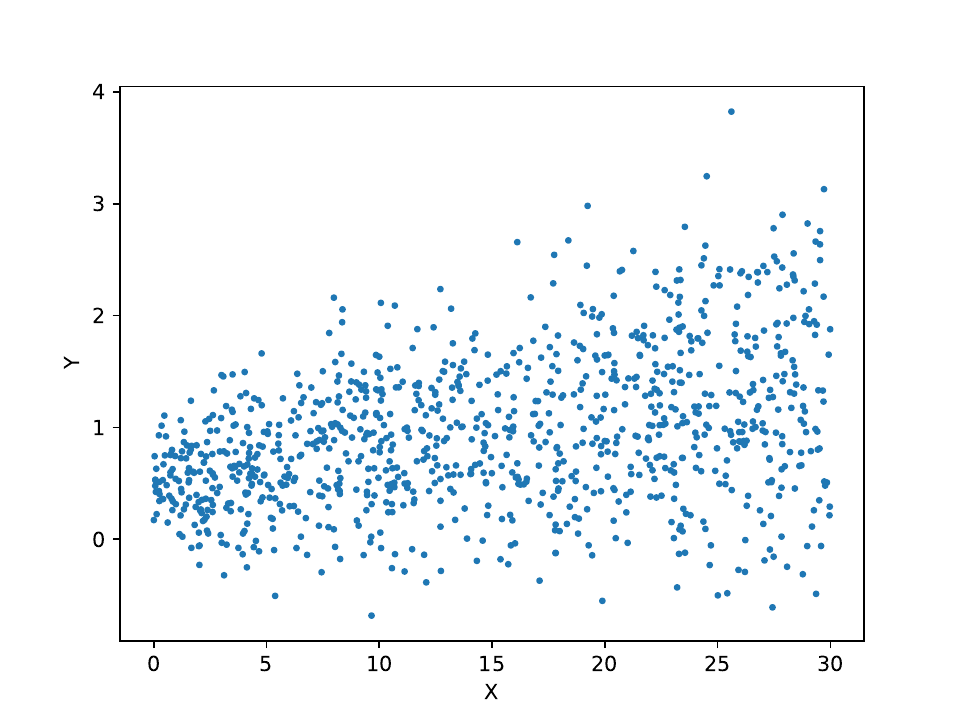}
    \caption{Scatterplot of data generated from FLD$(0.5, 0.2, 0.3 x)$, where $x\in(0,30)$}
    \label{fig:scatter_plot_fld}
\end{figure}
\end{minipage}
\hfill
\begin{minipage}[b]{0.45\textwidth}
\begin{figure}[H]
    \centering
    \includegraphics[width=\textwidth]{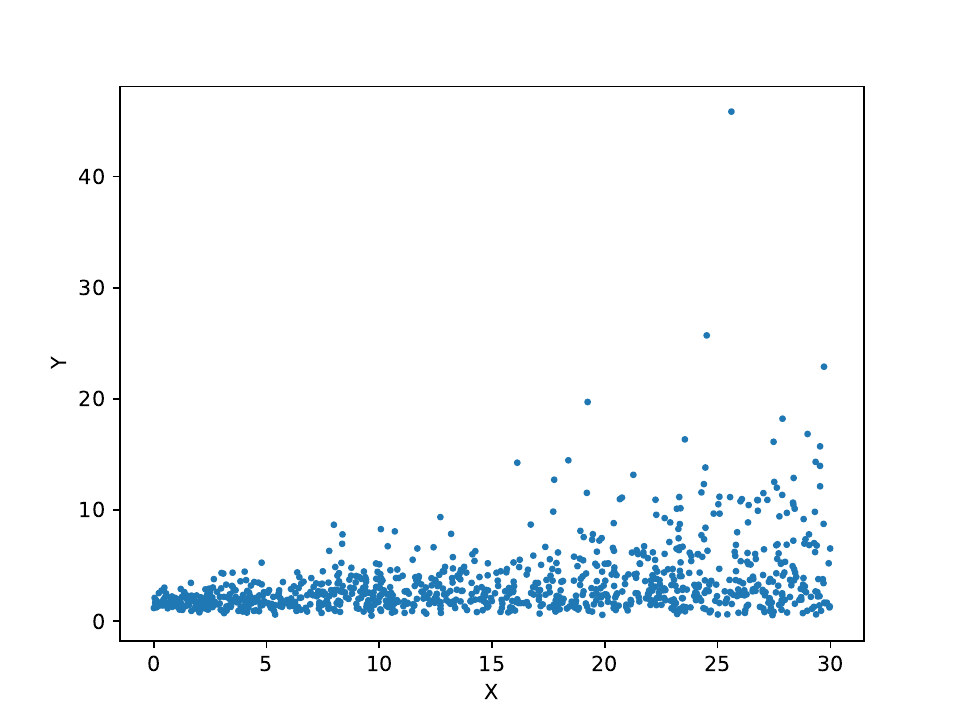}
    \caption{Scatterplot of the exponentiated data drawn from FLD depicted in Figure \ref{fig:scatter_plot_fld}}
    \label{fig:scatter_plot_exp_fld}
\end{figure}
\end{minipage}
\par\vspace{1ex}

%%%%%%%%%%%%%%%%%%%%%%%%%%%%%%%%%%%%%%%%%%%%%%%%%%%%%%%%%%%%%%%%%%%%%%%%%%%%%%%%%%%%%%%%%%%%%%%%%%%%%
\subsection{Other methods of obtaining estimators of the coefficients of the quantile regression}\label{sec:other_methods}
%%%%%%%%%%%%%%%%%%%%%%%%%%%%%%%%%%%%%%%%%%%%%%%%%%%%%%%%%%%%%%%%%%%%%%%%%%%%%%%%%%%%%%%%%%%%%%%%%%%%%

Bassett and Koenker \cite{Bassett1982} proposed estimating the conditional quantile function using a~linear interpolation between OQR estimators $\hat{\beta}^{OQR}(p_j)$, obtained for every $j\in\{1,2,\ldots, m-1\}$.
Such an estimator might be a~nonmonotonic function.
We denote plug-in estimators of the conditional indices based on this approach as BK.

Several methods to obtain monotonic estimators of quantile regression curves,
or more precisely linear quantile regression coefficients, have been proposed in the literature (see \cite{Wu2009} or \cite{Bondell2010} and references therein).
In the following, we describe those of them that were considered in comparisons with our proposed estimators obtained based on isotonic regression in the simulation study.
In Sections \ref{sec:cqr}, \ref{sec:wl}, and \ref{sec:brw}, $\hat{\beta}$ denotes CQR, WL1, and BRW estimators, respectively.

% \subsubsection{Linear interpolation between OQR estimators (BK)}\label{sec:bk}
% % \color{red}
% % opis BK:
% Bassett and Koenker \cite{Bassett1982} proposed estimating the conditional quantile function using a~linear interpolation between OQR estimators $\hat{\beta}^{OQR}(p_j)$, obtained for every $j\in\{1,2,\ldots, m-1\}$.
% Such an estimator might be a~nonmonotonic function.
% We denote plug-in estimators of the conditional indices based on this approach as BK.
% \color{black}

\subsubsection{Constrained quantile regression (CQR)}\label{sec:cqr}
%%%%%%%%%%%%%%%%%%%%%%%%%%%%%%%%%%%%%%%%%%%%%%%%%%%%%%%%%%%%%%%%%%%%%%%%%%%%%%%%%%%%%%%%%%%%%%%%%%%%%

Koenker and Ng \cite{Koenker2005a} described another approach that can be used to obtain the monotonicity of estimators of functions $\beta_i$ for any $i\in\{0, \ldots, d\}$.
Their algorithm allows estimating the coefficient $\beta$ for a~certain quantile order $p_j$ with some constraints on $\beta$.
For example, we may impose for each $i\in\{0, \ldots, d\}$ that the estimator $\hat{\beta}_i(p_j)$ of the coefficient $\beta_i(p_j)$ is larger than or equal to $\hat{\beta}_i(p_{j-1})$, for every $p_j$.
In a~similar stepwise method proposed by Wu and Liu, it is recommended to start by estimating (with the OQR method) the coefficients for $p=0.5$, since the variance of the OQR estimator of $\beta_i(p)$ is usually the lowest for $p=0.5$ (see \cite{Wu2009}).
Denote by $j^{*}$ a~number such that $p_{j^{*}}=0.5$.
After estimating $\hat{\beta}_i(p_{j^{*}})$, we estimate $\hat{\beta}_i(p_{j^{*}+1})$ and $\hat{\beta}_i(p_{j^{*}-1})$ with constraints that
$\hat{\beta}_i(p_{j^{*}+1}) \geq \hat{\beta}_i(p_{j^{*}})$ and
$\hat{\beta}_i(p_{j^{*}}) \geq \hat{\beta}_i(p_{j^{*}-1})$ and continue this procedure until we reach $\hat{\beta}_i(p_{1})$ and $\hat{\beta}_i(p_{m-1})$.
This approach may have a~potential risk of error propagation, that is, if the estimator of the median has a~large error, the other quantiles can also be estimated with large errors, since they strongly depend on the estimator of the `previous' quantile. 

\subsubsection{Stepwise quantile regression (WL1)}\label{sec:wl}
%%%%%%%%%%%%%%%%%%%%%%%%%%%%%%%%%%%%%%%%%%%%%%%%%%%%%%%%%%%%%%%%%%%%%%%%%%%%%%%%%%%%%%%%%%%%%%%%%%%%%

Wu and Liu \cite{Wu2009} proposed in 2009 two different stepwise methods that lead to noncrossing quantile regression lines.
The idea is to start with estimating the regression line for the median and then estimate step by step each quantile regression line by finding 
\begin{align*}
\hat{\beta}({p_{j+1}}) = \argmin_{\beta} \sum_{k=1}^{n} \rho_{p_{j+1}} \left( z_k  - \beta^{\top} x_k \right),
\end{align*}
with a~condition that
\begin{align*}
\hat{\beta}^{\top}({p_{j}}) x + \delta_0 \leq \beta^{\top}x, \quad \forall x \in \{0,1\}^d,
\end{align*}
where $\hat{\beta}(p_{j+1})$ and $\hat{\beta}(p_{j})$ are the estimators of the coefficients of quantile regression lines estimated for ${p_{j+1}}$ and ${p_{j}}$, respectively, with ${p_{j+1}}={p_{j}}+\frac{1}{m}$.
The parameter $\delta_0$ is chosen to be small and is given only to ensure that the subsequent coefficients are not equal to each other. 
Similarly, the coefficients are estimated for quantiles lower than the median. 
This approach was denoted by Wu and Liu as the naïve constrained estimation scheme.
They also proposed improved constrained estimates (which can be denoted by WL2) that are more efficient, as they have a~reduced number of constraints (see Section 3.2 in \cite{Wu2009}).
Hence, it can be a~good alternative, when the dimension $d$ of the data set is large.
Note that in this algorithm, each covariate is rescaled to an interval $[0,1]$.

\subsubsection{Simultaneously estimated noncrossing quantile regression (BRW)}\label{sec:brw}
%%%%%%%%%%%%%%%%%%%%%%%%%%%%%%%%%%%%%%%%%%%%%%%%%%%%%%%%%%%%%%%%%%%%%%%%%%%%%%%%%%%%%%%%%%%%%%%%%%

Bondell et al. \cite{Bondell2010} proposed a~method to estimate simultaneously noncrossing regression lines of quantile regression for multiple quantiles. 

% The idea of estimating the coefficients ${\beta}(p_j)$ for $j\in\{1, \dots, m-1\}$ is to find a~solution to an optimisation problem given by
% \begin{align*}
%     (\hat{\beta}(p_1), \hat{\beta}(p_2),\ldots,\hat{\beta}(p_{m-1})) = \arg \min_{\beta} \sum_{j=1}^{m-1} w(p_j) \sum_{i=1}^{n} \rho_{p_j} \left( y_i - \beta^{\top}_{p_j} z_i \right)
% \end{align*}
% subject to 
% \begin{align*}
% \beta^{\top}(p_j) z_{i} \geq \beta^{\top}(p_{j-1})z_{i} , \quad i\in \{1,\ldots,n\}, \quad j \in \{2, \dots, m-1\},
% \end{align*}
% where $w(p)$ is a~weight function and can be equal to 1~for each $p$ and $z_{i}=(1,x_{i1},x_{i2},\ldots,x_{id})$. 
% This type of problem can be solved using linear programming methods. 
% Here $\mathcal{D}$ is a~$d$-dimensional convex hull of the points given by a~sample
% ${\bf X}=(X_{1},\ldots,X_{n})$. 
The idea of estimating the coefficients ${\beta}(p_j)$ for $j\in\{1, \dots, m-1\}$ is to find a~solution to an optimisation problem given by
\begin{align*}
    \left(\hat{\beta}(p_1), \hat{\beta}(p_2),\ldots,\hat{\beta}(p_{m-1})\right) = \argmin_{\beta} \sum_{j=1}^{m-1} w(p_j) \sum_{k=1}^{n} \rho_{p_j} \left( z_k - \beta^{\top}(p_j) x_k \right)
\end{align*}
subject to 
\begin{align*}
\beta^{\top}(p_j) x \geq \beta^{\top}(p_{j-1})x, \quad \forall x \in \mathcal{D}, \quad j \in \{2, \dots, m-1\},
\end{align*}
where $w(p)$ is a~weight function and can be equal to 1~for each $p$.
% Here $x^{*}=(1,x_1,x_2,\ldots,x_d)$, for each $x=(x_1,x_2,\ldots,x_d) \in \mathcal{D}$, where $\mathcal{D}$ is a~convex~$d$-dimensional hull of the realisations of
% $X_{1},\ldots,X_{n}$. 
Here, $\mathcal{D}$ is a~convex~$d$-dimensional hull of the realisations of $X_{1},\ldots,X_{n}$. 
This type of problem can be solved using linear programming methods. 
The construction allows the quantile regression lines to cross for vectors $x$ that are outside of $\mathcal{D}$. 
This approach is an improvement compared to the stepwise method of Wu and Liu in terms of computation time.
% , but, as can be seen in the simulation study in Section \ref{sec:simulation_study}, it leads to plug-in estimators of conditional inequality measures with similar or larger MSE.

\subsubsection{Implementation of the quantile regression methods}\label{sec:technical_details}

The estimators $\hat{\beta}^{OQR}(p_j)$ of $\beta(p_j)$, for each $p_j$ were computed using the \textit{QuantileRegressor} from the Python module \textit{sklearn.linear\_model}.
% The measures were estimated for the exponent of the samples drawn from the flattened logistic distribution (FLD), described below.
Isotonic regression is implemented in specific libraries in both Python and R.
Here, the \textit{IsotonicRegression} class from \textit{sklearn.isotonic} module in Python was used.
Constrained quantile regression, needed to determine the estimates known from the literature and described in Section~\ref{sec:other_methods},
was performed in~R~using \textit{rq} function from \textit{quantreg} package with the option \textit{method="fnc"}.
The estimators $\hat{\beta}^{AQR}(p_j)$ were obtained using our Python-written implementation, using the function \textit{scipy.optimize.minimize}.
A~Python package \textit{cvxpy} was used to find estimates in BRW and WL1 methods. 
Similarly as in the case of methods IOQR and IAQR, the conditional quantile function obtained using BK, BRW, WL1 and CQR is constructed by plugging into a~formula \eqref{eq:assumption_ln} an estimator of a~function $\beta = (\beta_0, \beta_1, \ldots, \beta_d)$, which is a~linear interpolation between estimators of coefficients $\beta_i$ (for every $i=0,1,\ldots,d$) computed for a~grid of $m-1$ points $0<p_1<\ldots<p_{m-1}<1$, which are the orders of quantiles.
The integrals needed to estimate the conditional indices $qZI$ and $qDI$ are computed numerically using Simpson's rule.

%%%%%%%%%%%%%%%%%%%%%%%%%%%%%%%%%%%%%%%%%%%%%%%%%%%%%%%%%%%%%%%%%%%%%%%%%%%%%%%%%%%%%%%%%%%%%%%%%%%%%
\subsection{Results of the simulation study}
%%%%%%%%%%%%%%%%%%%%%%%%%%%%%%%%%%%%%%%%%%%%%%%%%%%%%%%%%%%%%%%%%%%%%%%%%%%%%%%%%%%%%%%%%%%%%%%%%%

We compare the finite sample accuracy of the plug-in estimators of the conditional $qZI$ and $qDI$ indices based on methods described in Section \ref{sec:measures} to obtain the conditional quantile function.
The BK method, described in Section \ref{sec:other_methods}, does not guarantee the monotonicity of $Q_x$ with respect to $p$.
It is guaranteed though by BRW, WL1, and WL2. 
Since for $d=1$ the WL2 method produces estimators very similar to the WL1 method, its results will not be presented in the description of the results.
The CQR, IOQR, and IAQR methods also guarantee the monotonicity of functions $\beta_i$ with respect to $p$, for each $i\in\{0,\ldots,d\}$.
% In case of IOQR and IAQR, we consider obtaining estimates of $\beta$ coefficients with ordinary quantile regression (IOQR) and with the smooth approximation ($f_{\tau}$) of the loss function (IAQR). 
The default value of the parameter $\tau$ used in these simulations to obtain the IAQR estimates was
\begin{equation}
    \hat{\tau} = \frac{1}{\sqrt{n}} \left[\widehat{IQR}(Y|X=\bar{x})\right],
\end{equation}
where $n$ is the number of observations, $\bar{x}$ is the sample mean of the covariate $X$ and $\widehat{IQR}(Y|X=x)$ is the estimator (obtained with quantile regression applied only to two quantiles of orders $0.25$ and $0.75$, without applying isotonic regression) of the interquartile range, i.e. the difference between the quantile of order $0.75$ and $0.25$, when the value of the covariate $X$ is equal to some $x$.
% \color{green}
% Although this value of $\tau$ is not optimal, it behaves quite well and is stable in all the cases considered in this simulation study.
% The optimal value of $\tau$ can be found, for example, using cross-validation.
% 
The value of $\tau$ was arbitrarily chosen and it is possible to obtain an estimator with a~lower MSE for a~different value of $\tau$.
Simpson's rule is used to numerically calculate the values of the plug-in estimators of the conditional inequality measures $qZI$ and $qDI$. For each set of parameters $5000$ Monte Carlo repetitions were performed with samples of sizes $n\in\{50, 100, 500, 1000\}$.

The exact true values of the indices in all cases considered in the simulation study are given in Figure \ref{fig:qzi_qdi_true_values}.
The higher values of the parameter $\gamma$ lead to a~larger range of values of the parameter $\kappa$ which indicates a~larger range of the inequality measures values, for example, for $\alpha=0.5, \beta=0.2$ if $\gamma=0.1$ then $qDI$ varies from $0.377$ to $0.5$, while if $\gamma=0.3$ then it varies from $0.387$ to $0.659$, when $x$ varies from 1 to 30. 
The range of inequality measures can also be increased by increasing the range of covariates $X$, which for the sake of feasibility of this simulation study was set to be between 0 and 30 for each set of parameters considered in this paper.

\begin{figure}
    \centering
    \includegraphics[width=\textwidth]{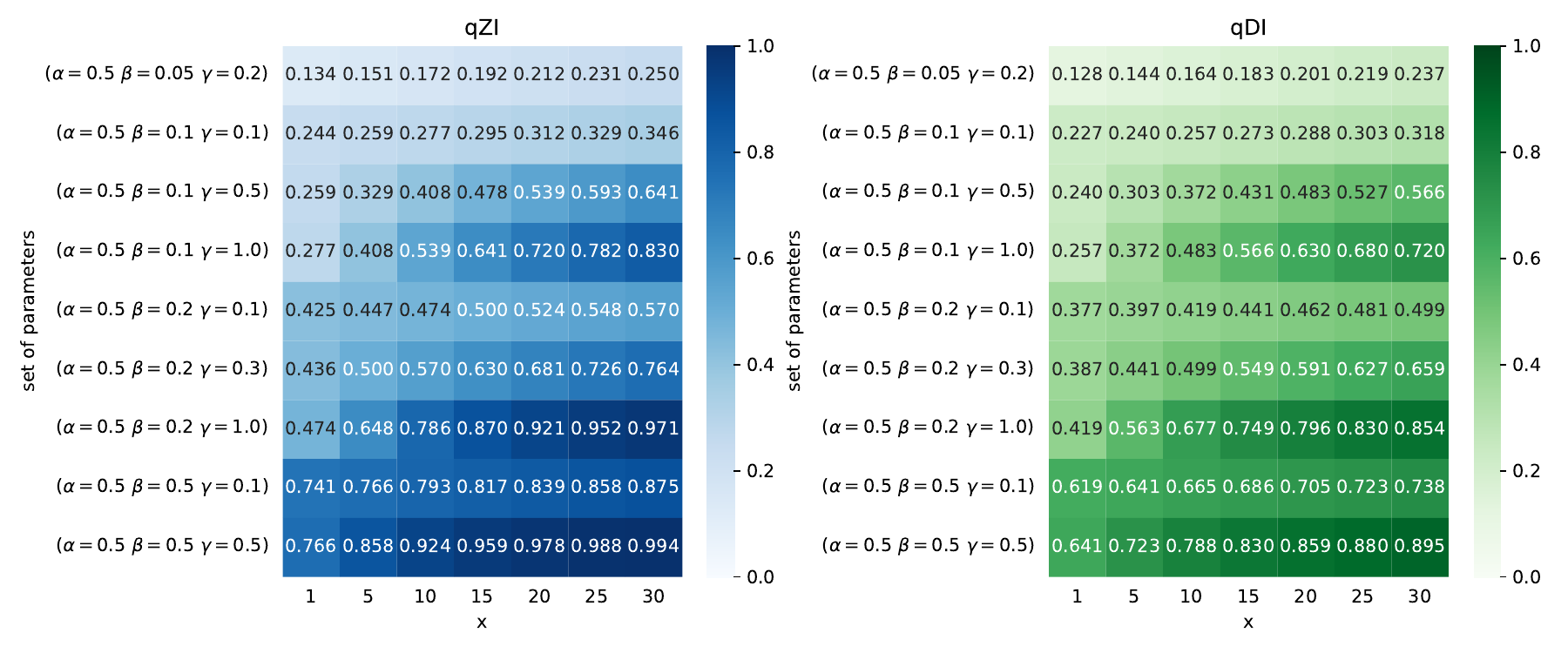}
    \caption{True values of the indices $qZI$ and $qDI$ of $EFLD$, for nine different sets of parameters and $x$ varying between 1 and 30}
    \label{fig:qzi_qdi_true_values}
\end{figure}

Figure \ref{fig:mse_heatmap} contains a~comparison of MSE of plug-in estimators of both conditional indices $qZI$ and $qDI$ obtained with six methods of estimating the conditional quantile function. 
For each combination of $n$ and the parameters $(\alpha,\beta,\gamma)$ of $FLD$, an element of the heatmap represents the ratio of MSE of a~particular method and the best method for that $n$ and $(\alpha,\beta,\gamma)$.
The order of the sets of parameters $(\alpha,\beta,\gamma)$ in each heatmap is the same as in Figure \ref{fig:qzi_qdi_true_values}.
% \color{red}
% Ten fragment mi tu nie pasuje i proponuję go przenieść dalej.
% The errors of the estimators, understood as the differences between the estimators of the conditional index and its true value, are illustrated  in the boxplots in Figures \ref{fig:errors_100_5_1_1} -- \ref{fig:errors_500_5_1_5}, while their MSEs are compared in Figures \ref{fig:mse_heatmap} and \ref{fig:mse_heatmap_values}.
% \color{black}

It can be seen in Figure \ref{fig:mse_heatmap} that for $n=50$ in many cases the estimators IAQR and CQR have a~lower MSE than other methods. 
Although in many cases the methods BRW, WL1 and BK give similar or slightly better results (in terms of MSE), for several cases they are significantly worse, specifically when the values of the estimated index are large for every $x$.
IOQR has slightly higher MSE than IAQR in most cases.
For $n=100$ the results are similar to those of $n=50$, although the differences between the methods are smaller.
For lower values of $\beta$, IAQR is preferred, while CQR has the lowest MSE for higher $\beta$.
The BRW, WL1, and BK methods have significantly higher MSE for high $x$, when the sample size is small. 
In Figure \ref{fig:errors_50_5_5_5} we can see that the results obtained with the IAQR and IOQR methods have fewer outliers and lower variance. 
Thus, in this cumbersome case where $n$ is small and the conditional concentration indices are high, these methods are more stable than their competitors.

For large $n$ (500 and 1000), the differences between the estimators are small in terms of MSE.
The differences between methods are more relevant for $x$ close to the boundary values (0 and 30), while they are much smaller in the middle of its range ($x=15$).
Figure \ref{fig:mse_heatmap_values} shows the MSE values for each method (except WL2), each combination of parameters considered in this study, and for sample sizes 50 and 100. 
For $n\in\{500,1000\}$ MSEs are relatively small and have similar values for all methods, thus we only display the MSE values for $n\in\{50,100\}$.
It can be seen in this figure that the heatmap showing MSE of the IAQR estimator contains fewer dark areas; hence it is less prone to have a~high level of MSE.

Figures \ref{fig:errors_100_5_1_1}, 
\ref{fig:errors_50_5_5_5} and \ref{fig:errors_500_5_1_5}
show boxplots of the differences between the true values of the indices $qDI$ (top) and $qZI$ (bottom) and their estimates obtained using IOQR, IAQR, BK, BRW, WL1, and CQR as previously defined for several combinations of sample size $n$ and parameters $(\alpha,\beta,\gamma)$ of $FLD$.
Similar figures for the other combinations of $n$ and $(\alpha,\beta,\gamma)$ can be found in a~GitHub repository (see Section Data and Code Availability). 
% \url{https://github.com/skfp/conditional_measures} \rt{Jesli dodasz sekcje Data and Code avaiblity to tutaj zamiast linku proponuję napisać c.f. Section Data and Code Availability.}.
The values on the $x$ axis indicate the value of the covariate $X$ for which conditional $qZI(Q_x)$ and $qDI(Q_x)$ were computed.

\begin{figure}
    \centering
    \includegraphics[width=\textwidth]{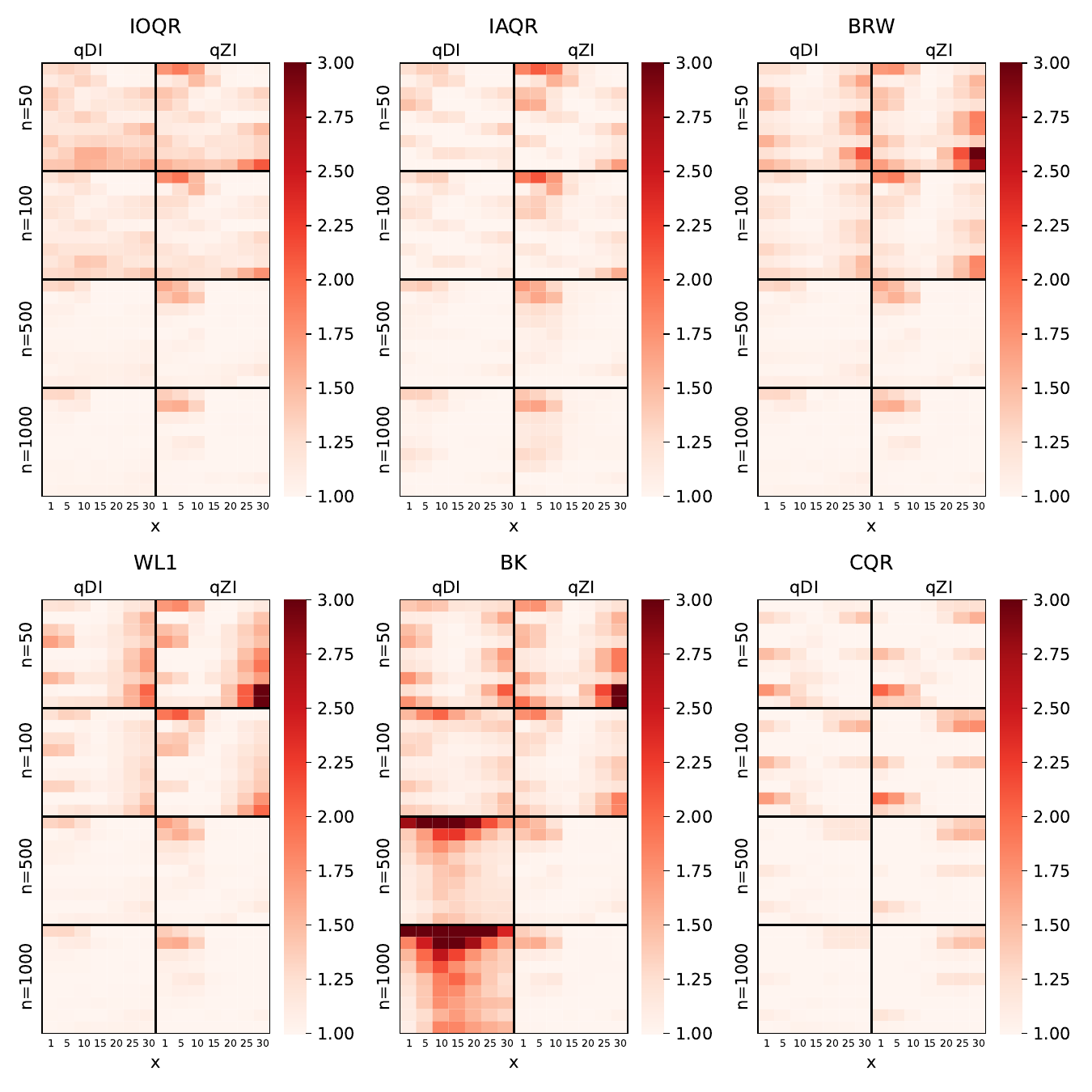}
    \caption{Ratio of MSE of the plug-in estimator of the conditional inequality measures of samples drawn from $EFLD$ for samples of size 50, 100, 500 and 1000 and for nine combinations of parameters of $EFLD$, obtained with a~given method, to the lowest MSE in each case. The order of the sets of parameters is the same as in Figure \ref{fig:qzi_qdi_true_values}. The values are clipped on the level of 3.
    }
    \label{fig:mse_heatmap}
\end{figure}

Figure \ref{fig:errors_100_5_1_1} shows an example of the~case when the CQR estimator underestimates the indices values when the parameter $\kappa$ (and consequently the inequality) is low, while it overestimates the inequality measures for high $\kappa$.
% Thus, CQR can be considered as a~less robust and less stable method than IOQR and IAQR.
The estimators based on the methods BK, BRW, and WL1 behave similarly to IOQR and IAQR and it is difficult to choose the best method based on the boxplots. 
Naturally, IAQR estimators tend to IOQR estimators when $\tau$ tends to 0, thus for larger $n$ and smaller $\widehat{IQR}$ (for $\bar{x}$) the IOQR and IAQR results will be almost identical.

\begin{figure}
    \centering
    \includegraphics[width=\textwidth]{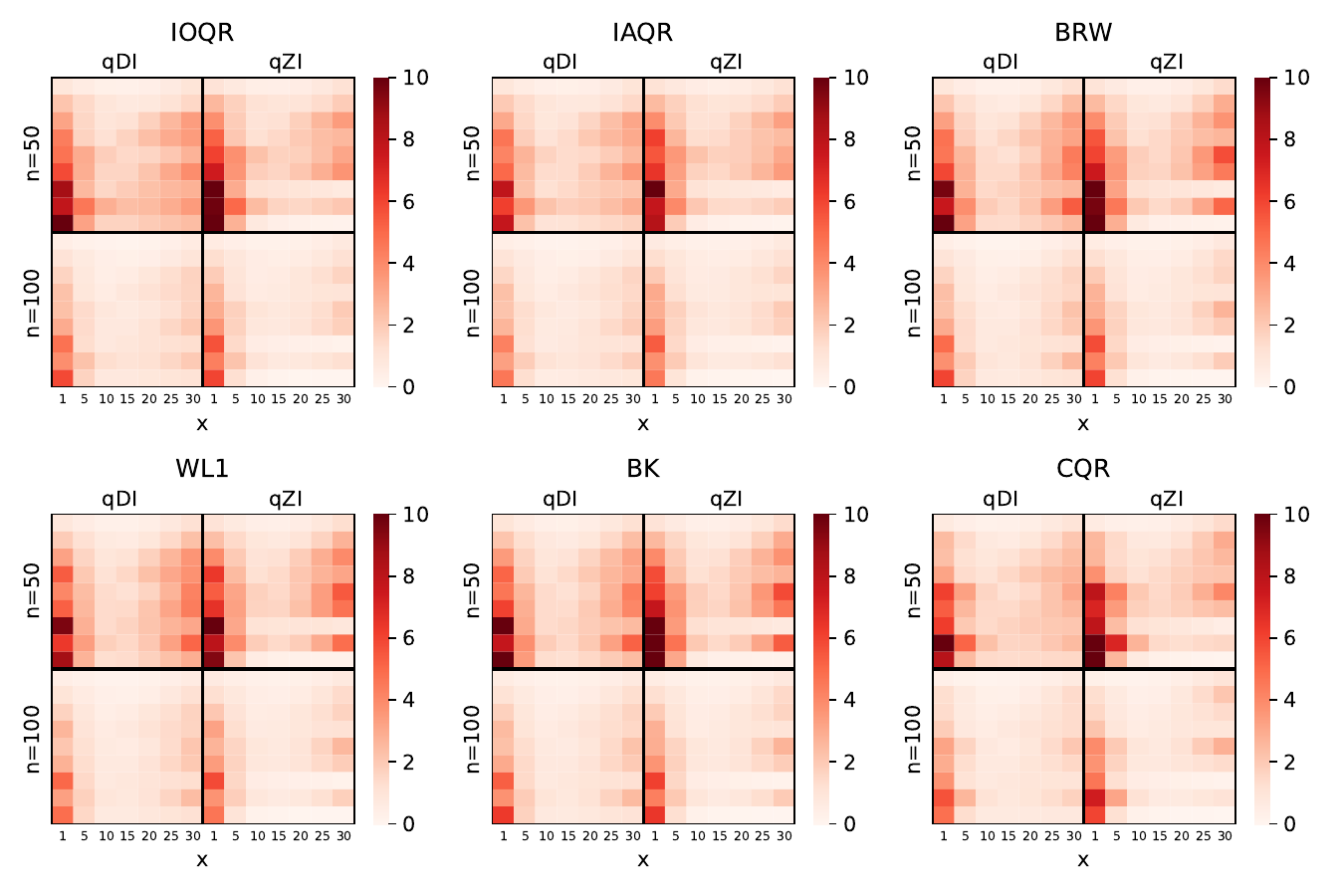}
    \caption{MSE of the estimators of conditional $qZI$ and $qDI$ of samples drawn from $EFLD$ for samples of size 50 and 100 and for nine combinations of parameters of $EFLD$. The order of the sets of parameters is the same as in Figure \ref{fig:qzi_qdi_true_values}.}
    \label{fig:mse_heatmap_values}
\end{figure}

For $n=100$ in most cases considered in this simulation study, the IAQR method gives a~relatively small MSE when estimating both $qZI$ and $qDI$. 
WL1 estimators of both inequality measures have a~small MSE for small $\kappa$, comparable to IAQR, but for larger $\kappa$ its MSE is larger than that of IAQR and similar to that of BRW.
However, CQR has the lowest MSE in the case when inequalities are large, but the range of the values of the parameter $\kappa$ is narrow.
For $n\in\{500,1000\}$, the IAQR estimator failed to converge in several cases, resulting in extremely inaccurate values of the estimators of the indices. However, it should be noted that this occurred in less than $0.1\%$ of the samples used in the simulation study.
% \rt{Czy wiemy dlaczego ten estymator ekspodowal? Jesli tak to moze to dodac. Ale to opcjonalne.}
For $n\in\{500,1000\}$, the BK estimator significantly overestimated conditional $qDI$, which can be seen in Figures \ref{fig:mse_heatmap} and \ref{fig:errors_500_5_1_5}.
Thus, any of the estimators based on noncrossing quantile regression lines is better than the reference method of Bassett and Koenker \cite{Bassett1982}.

\begin{figure}[H]
	\centering
	\includegraphics[width=\textwidth]{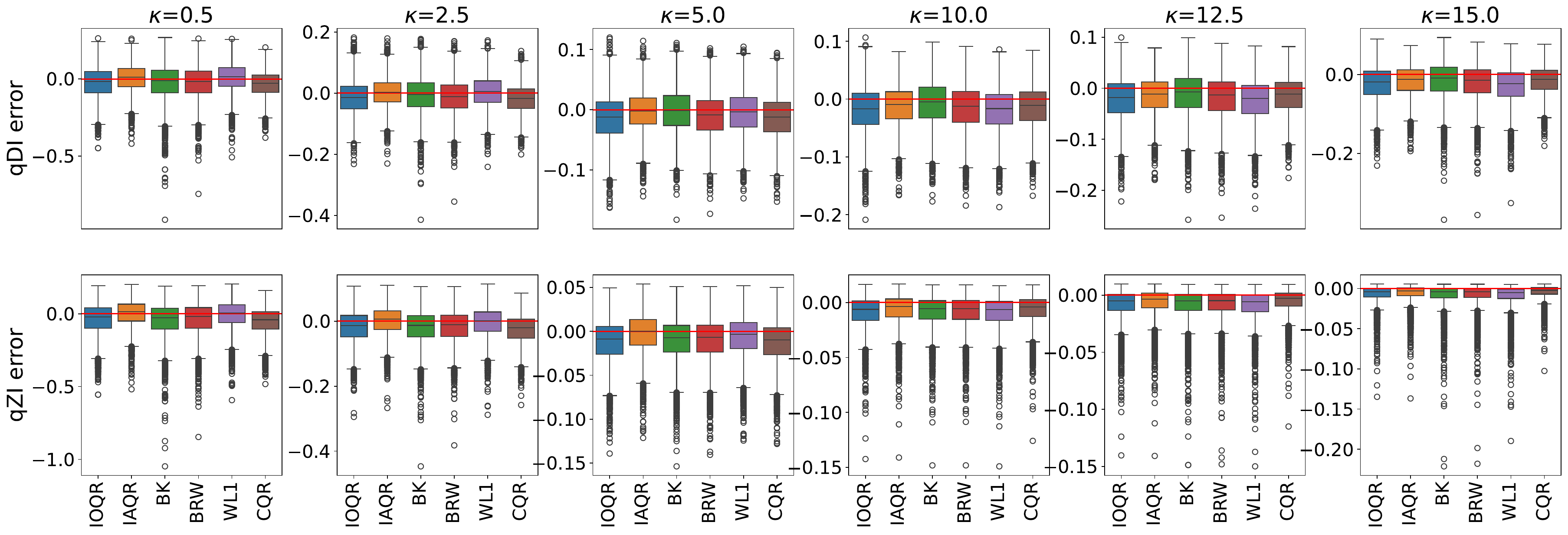}
	\caption{Errors of the estimators of indices $qZI$ and $qDI$ for $n=50$, $\alpha=0.5, \beta=0.5$ and $\kappa$ ranging between 0.5 and 15}
	\label{fig:errors_50_5_5_5}
\end{figure}

\begin{figure}[H]
	\centering
	\includegraphics[width=\textwidth]{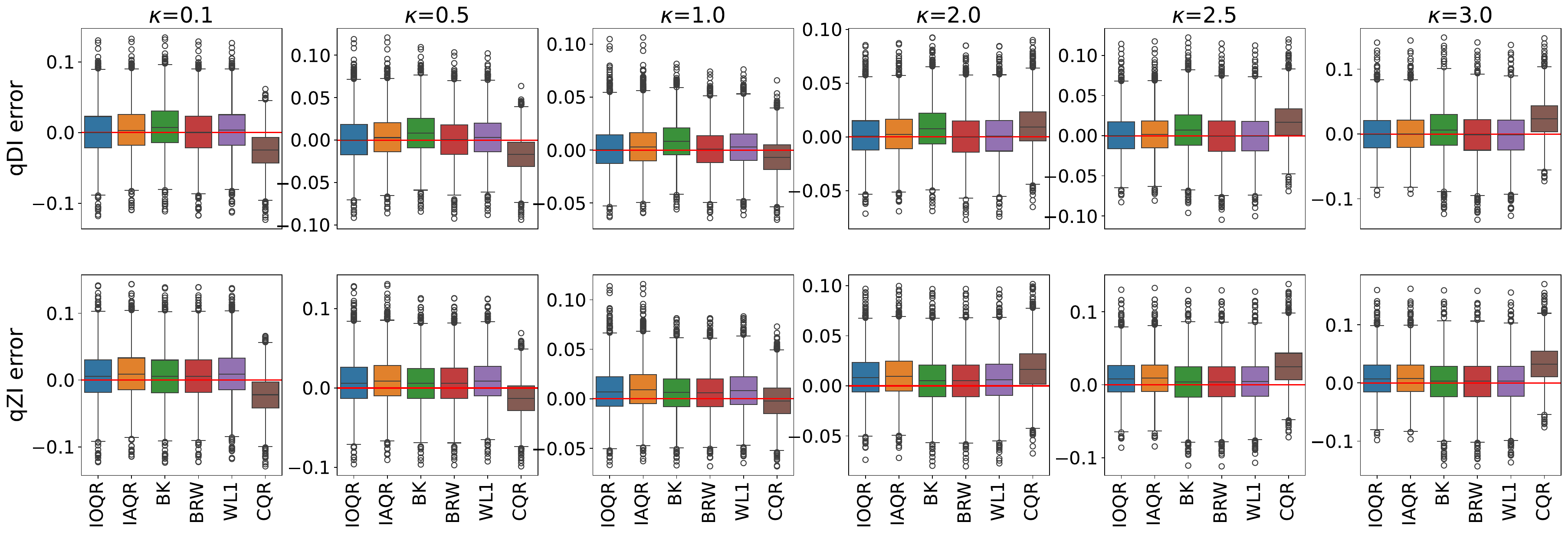}
	\caption{Errors of the estimators of indices $qZI$ and $qDI$ for $n=100$, $\alpha=0.5, \beta=0.1$ and $\kappa$ ranging between 0.1 and 3}
	\label{fig:errors_100_5_1_1}
\end{figure}

\begin{figure}[H]
	\centering
	\includegraphics[width=\textwidth]{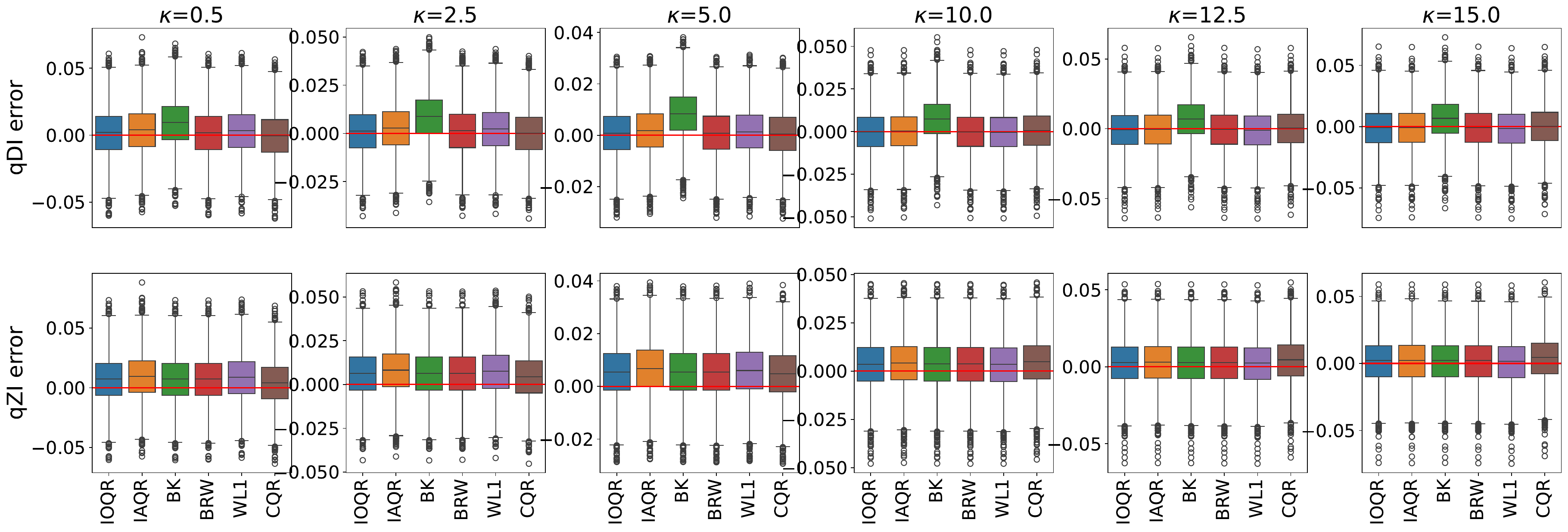}
	\caption{Errors of the estimators of indices $qZI$ and $qDI$ for $n=500$, $\alpha=0.5, \beta=0.1$ and $\kappa$ ranging between 0.5 and 15}
	\label{fig:errors_500_5_1_5}
\end{figure}

%%%%%%%%%%%%%%%%%%%%%%%%%%%%%%%%%%%%%%%%%%%%%%%%%%%%%%%%%%%%%%%%%%%%%%%%%%%%%%%%%%%%%%%%%%%%%%%%%%
\section{Salary data analysis}\label{sec:real_data}
%%%%%%%%%%%%%%%%%%%%%%%%%%%%%%%%%%%%%%%%%%%%%%%%%%%%%%%%%%%%%%%%%%%%%%%%%%%%%%%%%%%%%%%%%%%%%%%%%%%

\begin{figure}
    \centering
    \includegraphics[width=\textwidth]{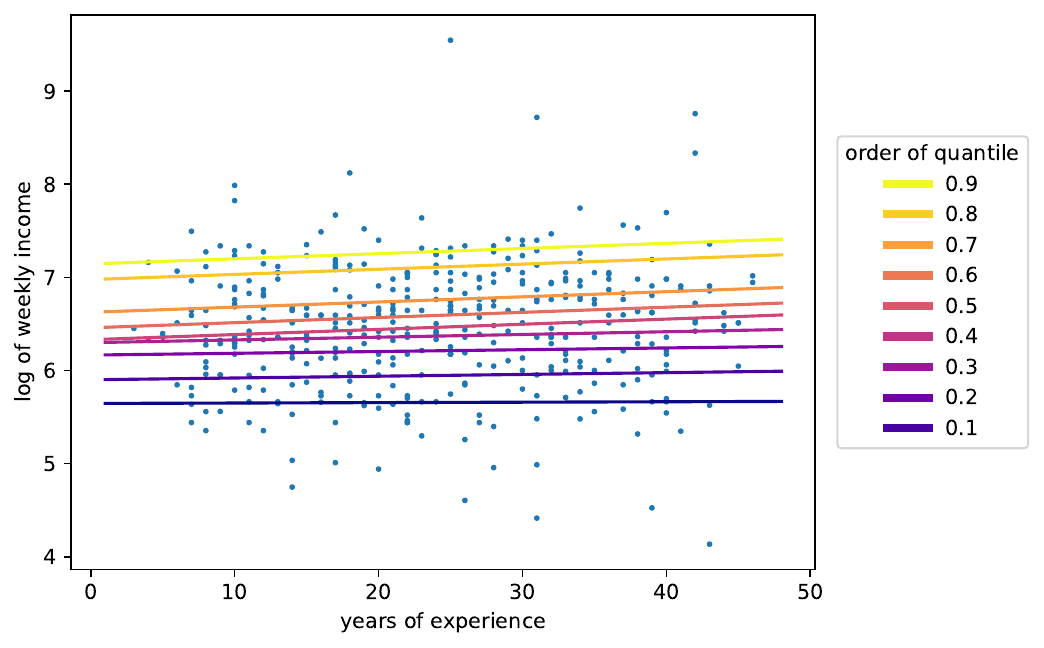}
    \caption{Plot of the logarithms of weekly wage of employees from Oklahoma with different working experience. The lines indicate nine conditional quantiles of orders $\{0.1, 0.2, \ldots, 0.9\}$ estimated with the IAQR method.}
    \label{fig:log_weekly_wage_quantiles_OK}
\end{figure}
An analysis of the data set with years of experience and the weekly salary of workers is described below. The original data set contains 29,501 records \cite{census2000}, but restricting attention to records from Oklahoma, we obtain a~subset of size 407.
We assume a~linear relation between the years of experience (\textit{exper}) and the logarithm of the weekly salary (\textit{lweekinc}).
He and Zhu \cite{He2003} proposed a~lack-of-fit test that can be applied to both linear and nonlinear quantile regression. 
This test is implemented in the \textit{Qtools} package in R in a~\textit{LOFTest} function.  
% However, for high-dimensional datasets, it is recommended to use a~test proposed by Dong et al. \cite{Dong2019}.
Since the test described by He and Zhu suggested that there is no reason to reject the linearity hypothesis for any of the 99 quantiles, we can proceed further with the assumption of a~linear model.

% In Section \ref{sec:properties}, it was assumed that $X$ follows an absolutely continuous distribution. 
% In this analysis, $X$ represents the years of experience, which is a~discrete random variable. 
% However, the methods described in this paper remain useful in this setting, and loosing this assumption does not substantially affect its practical application. 
% Moreover, many asymptotic results derived for the continuous case often serve as reasonable approximations for discrete variables, especially when the range of $X$ is large, which is the case here. 
% Therefore, it is justified to apply the IOQR or IAQR estimators of conditional inequality measures in this case. 

The IAQR was used to estimate the coefficients $\beta_0$ and $\beta_1$ for 99 quantiles, which were later used to estimate $qZI$ and $qDI$ conditioned on the years of experience of the surveyed workers.
Figure \ref{fig:log_weekly_wage_quantiles_OK} presents the scatterplot of \textit{lweekinc} and \textit{exper} with 9 estimated conditional quantiles of orders $(0.1, 0.2, \ldots, 0.9)$.
It can be seen there that the quantiles are estimated in such a~way that quantile-crossing is prevented, since the slope of the lines is higher for higher quantiles.

The conditional $qZI$ (left panel) and $qDI$ (right panel) of the weekly wage depending on the experience are presented in Figure \ref{fig:four_states_qZI_qDI} with the lines
for four different states of the USA: Maryland, Oklahoma, Oregon and Tennessee. The choice of these states was made to ensure the linearity of approximately 95\% of the conditional quantiles and to have sample sizes close to 500 for each state.

It can be seen that the choice of the inequality measure matters, since the analysis of $qZI$ may lead to slightly different conclusions than the analysis of $qDI$.
Conditional $qZI$ in Oklahoma and Oregon has similar values, although for workers with less than 25 years of experience, it is slightly higher in Oregon, while for those with more than 25 years of experience, it is slightly higher in Oklahoma.
In contrast, $qDI$ is higher in Oregon than in Oklahoma in the whole range of tenure considered in this analysis.
The conditional indices $qZI$ and $qDI$ vary much more in Maryland and Tennessee than in Oklahoma and Oregon.
All eight curves are increasing, although the growth of the $qDI$ indices for Oregon is barely noticeable.
This analysis shows that the character of income inequalities and their relation to the employee experience vary between states.

One use of this analysis is to identify states where income inequalities are more explicit among workers with varying levels of experience.
This information can be valuable for policy makers and economists in understanding the dynamics and characteristics of income distribution within different regions.
For example, if a~certain state has a~steep increase in inequality measures with experience, it may indicate that experienced workers are disproportionately benefiting from wage increases compared to their less experienced colleagues.

Such analyses may help identify potential areas for further research, such as exploring the underlying factors contributing to income inequality in specific states or industries. 
In general, the results of the analysis of the conditional indices curves, such as the ones in Figure \ref{fig:four_states_qZI_qDI}, can provide valuable insight into the complex relationship between experience and income inequality, guiding efforts to create more equitable economic environments.

\begin{figure}[H]
    \centering
    \includegraphics[width=\textwidth]{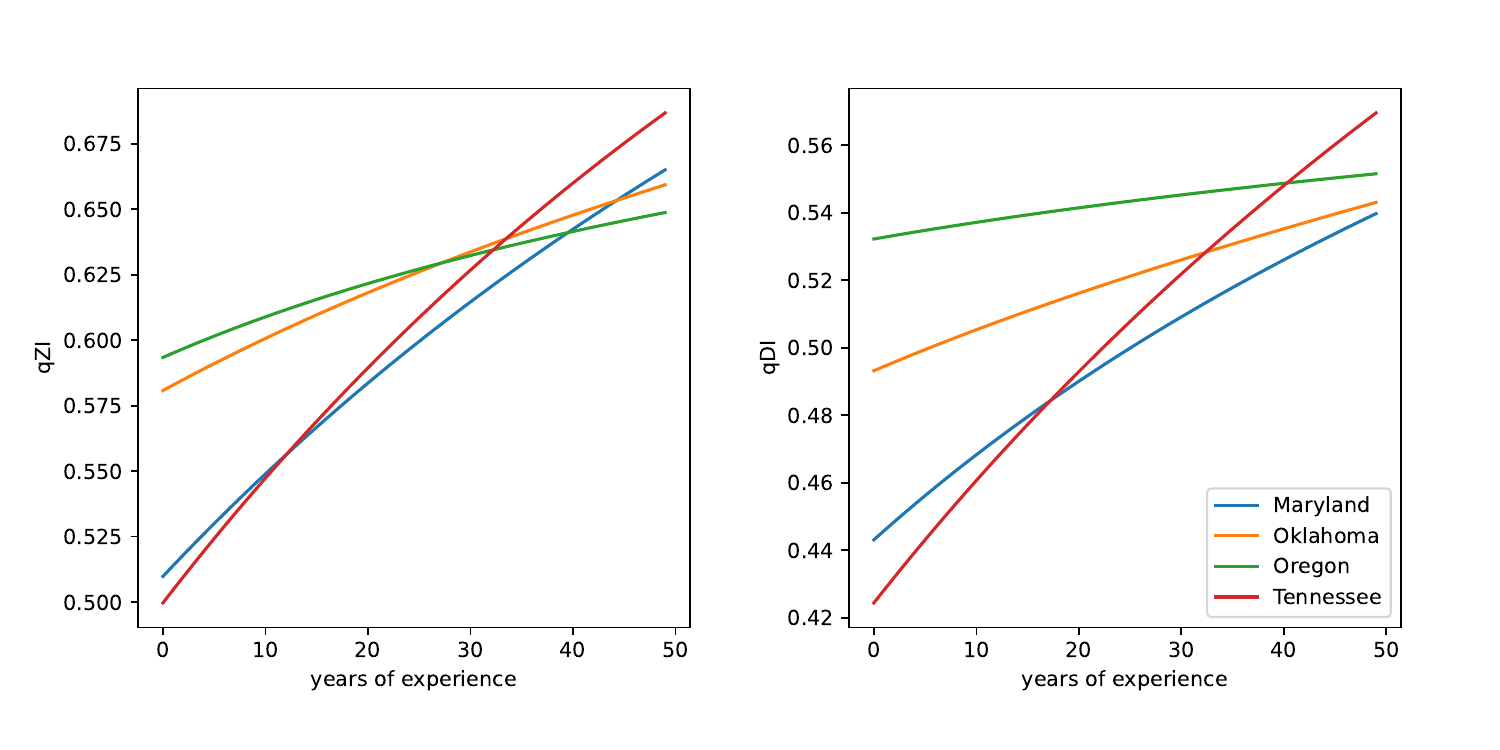}
    \caption{Estimated curves of $qZI(Q_{x})$ (left panel) and $qDI(Q_{x})$ (right panel) with respect to $x$ -- years of experience for Maryland, Oklahoma, Oregon and Tennessee,
    obtained by IAQR method}
    \label{fig:four_states_qZI_qDI}
\end{figure}

%%%%%%%%%%%%%%%%%%%%%%%%%%%%%%%%%%%%%%%%%%%%%%%%%%%%%%%%%%%%%%%%%%%%%%%%%%%%%%%%%%%%%%%%%%%%%%%%%%%%%%%%
%%%%%%%%%%%%%%%%%%%%%%%%%%%%%%%%%%%%%%%%%%%%%%%%%%%%%%%%%%%%%%%%%%%%%%%%%%%%%%%%%%%%%%%%%%%%%%%%%%%%%%%%
\section{Conclusions and some prospects}\label{sec:conclusion}
%%%%%%%%%%%%%%%%%%%%%%%%%%%%%%%%%%%%%%%%%%%%%%%%%%%%%%%%%%%%%%%%%%%%%%%%%%%%%%%%%%%%%%%%%%%%%%%%%%%%%%
%%%%%%%%%%%%%%%%%%%%%%%%%%%%%%%%%%%%%%%%%%%%%%%%%%%%%%%%%%%%%%%%%%%%%%%%%%%%%%%%%%%%%%%%%%%%%%%%%%%%%%%%

In this paper, we proposed conditional inequality curves and measures, and also the curve of conditional inequality measures, which illustrates the change of inequalities with respect to a~continuous variable.
We also proposed estimators of the above curves and measures based on quantile regression and isotonic regression, which lead to estimators with the desired properties.
In addition, a~new approximation of the loss function was proposed and used to determine the estimates of the quantile regression coefficients.
As shown by simulation studies, this approximation leads to more accurate estimates than the standard approach used in quantile regression.
The estimation methods proposed in the paper were applied to the analysis of real data (incomes from several US states), in particular to determine the estimated curve of conditional inequality measures
with respect to years of experience.

In future work, we plan to adopt the approach proposed by Farcomeni and Geraci \cite{Farcomeni2024} to estimate the curves and measures considered in the paper.
Farcomeni and Geraci \cite{Farcomeni2024} proposed a~class of regression models involving the ratio of two quantiles of the dependent variable.
It is worth noting that their proposal concerns models for single point of the conditional curves we consider.
We think that their proposal can be adapted to estimate entire conditional curves and inequality measures.

%%%%%%%%%%%%%%%%%%%%%%%%%%%%%%%%%%%%%%%%%%%%%%%%%%%%%%%%%%%%%%%%%%%%%%%%%%%%%%%%%%%%%%%%%%%%%%%%%%%%%%%%
%%%%%%%%%%%%%%%%%%%%%%%%%%%%%%%%%%%%%%%%%%%%%%%%%%%%%%%%%%%%%%%%%%%%%%%%%%%%%%%%%%%%%%%%%%%%%%%%%%%%%%%%
\section{Appendix}
%%%%%%%%%%%%%%%%%%%%%%%%%%%%%%%%%%%%%%%%%%%%%%%%%%%%%%%%%%%%%%%%%%%%%%%%%%%%%%%%%%%%%%%%%%%%%%%%%%%%%%%%
%%%%%%%%%%%%%%%%%%%%%%%%%%%%%%%%%%%%%%%%%%%%%%%%%%%%%%%%%%%%%%%%%%%%%%%%%%%%%%%%%%%%%%%%%%%%%%%%%%%%%%%%

\subsection{Proof of Theorem \ref{thm:IOQR_conv}}
%%%%%%%%%%%%%%%%%%%%%%%%%%%%%%%%%%%%%%%%%%%%%%%%%%%%%%%%%%%%%%%%%%%%%%%%%%%%%%%%%%%%%%%%%%%%%%%%%%%%%%%%
The following lemma will be necessary to prove the theorem.

\begin{lemma}\label{lem:lower_sup}
Let $\bigl(((p_1, \hat{\xi}(p_1))$, $\ldots,$ $(p_{m-1}, \hat{\xi}(p_{m-1}))\bigr)$ be a~sequence of $m-1$ points such that $p_1<\ldots<p_{m-1}$, where $m>2$. 
Let $\tilde{\xi}(p_1),$ $\ldots,$ $\tilde{\xi}(p_{m-1}))$ be a~nondecreasing sequence of points which minimises
    \begin{align*}%\label{e:g_isotonic}
        \sum_{j=1}^{m-1} \left[\hat{\xi}(p_j)-\tilde{\xi}(p_j)\right]^2.
    \end{align*}
If $\xi$ is a~nondecreasing function on $[0,1]$, we have
    \begin{align}\label{e:sup_tilde_hat_inequality}
        \max_{j\in\{1,\ldots,m-1\}}\left|\tilde{\xi}(p_j)-{\xi}(p_j)\right| \leq
        \max_{j\in\{1,\ldots,m-1\}}\left|\hat{\xi}(p_j)-{\xi}(p_j)\right|.
    \end{align}
\end{lemma}

\begin{proof}
    Denote by $A = \{ p_j\colon \tilde{\xi}(p_j) = \hat{\xi}(p_j) , j\in\{1,2,\ldots,m-1 \} \}$
    and $B=B_1 \cupdot \ldots \cupdot B_r$, where $r$ is a~number of such disjoint subsets $B_i$, where for each $i$ we have a constant $c_i$ such that $\tilde{\xi}(p_{j}) = c_i$, for every $p_{j}\in B_i$.
    Let the subset $B_i$ be of size $k+1$, i.e. $B_i=\{p_{i_1}, p_{i_1+1},\ldots,p_{i_1+k}\}$. From Corollary 2 in \cite{Best1990}, we know that $\hat{\xi}(p_{i_1}) > c_i > \hat{\xi}(p_{i_1+k})$. 
    We will show that for each $B_i$ we have
    \begin{equation}\label{eq:hat_tilde_maximum}
        \max_{p_j\in B_i}\left|\hat{\xi}(p_{j})-{\xi}(p_{j})\right| \geq
        \max_{p_j\in B_i}\left|\tilde{\xi}(p_{j})-{\xi}(p_{j})\right|.
    \end{equation}
The values $\xi(p_{{i_1}})$ and $\xi(p_{{i_k}})$ can be located in several ways with respect to the points $\hat{\xi}(p_{{i_1}})$ and $\hat{\xi}(p_{{i_k}})$. We will show that in each of the possible cases \eqref{eq:hat_tilde_maximum} holds.
The cases are as follows:
\begin{enumerate}[label=(\roman*)]
\item ${\xi}(p_{{i_1}}) \geq \hat{\xi}(p_{{i_1}})$,
\item $\hat{\xi}(p_{{i_1}}) \geq {\xi}(p_{{i_1}})$ and 
${\xi}(p_{{i_k}}) \geq \hat{\xi}(p_{{i_k}})$,
\item $\hat{\xi}(p_{i_k}) \geq \xi(p_{i_k})$.
\end{enumerate}
In case (i), $\xi(p_{i_1}) \geq \hat{\xi}(p_{{i_1}})$ and the monotonicity of $\tilde{\xi}$ and ${\xi}$ imply that
\begin{equation*}
    {\xi}(p_{i_k}) \geq {\xi}(p_{i_1}) \geq \hat{\xi}(p_{i_1})
    > c_i > \hat{\xi}(p_{i_k}).
\end{equation*}
Hence $\left|\hat{\xi}(p_{i_1+k}) - {\xi}(p_{i_1+k})\right| \geq
\left|c_i - {\xi}(p_{i_1+k})\right| = \max_{p_j\in B_i}\left|\tilde{\xi}(p_{j}) - {\xi}(p_{j})\right|$. \\
Case (ii) can be divided into three subcases: \\
a) When $\hat{\xi}(p_{i_1}) > \xi(p_{i_1}) \geq c_i$, we have
\begin{align*}
    \max_{p_{i_j}\in B_i}\left| c_i - {\xi}(p_{i_j})\right| = \left| c_i - {\xi}(p_{i_1+k})\right| \leq
\left| \hat{\xi}(p_{i_1+k}) - {\xi}(p_{i_1+k})\right|
\leq \max_{p_{i_j}\in B_i} \left| \hat{\xi}(p_{i_j}) - {\xi}(p_{i_j})\right|.
\end{align*}
b) When ${\xi}(p_{i_k}) \geq c_i \geq \xi(p_{i_1})$, we have
\begin{align*}
    \max_{p_{i_j}\in B_i}| c_i - {\xi}(p_{i_j})| = 
    | c_i - {\xi}(p_{i_1})| \vee | c_i - {\xi}(p_{i_k})|,
\end{align*}
but we also have
\begin{align*}
    \left| c_i - {\xi}(p_{i_1})\right| \leq  \left| \hat{\xi}(p_{i_1}) - {\xi}(p_{i_1})\right|
    \qquad \text{and} \qquad
    \left| c_i - {\xi}(p_{i_1+k})\right| \leq  \left| \hat{\xi}(p_{i_1+k}) - {\xi}(p_{i_1+k})\right|,
\end{align*}
hence
\begin{align*}
    \max_{p_{i_j} \in B_i}\left| c_i - {\xi}(p_{i_j})\right| \leq
    \left| \hat{\xi}(p_{i_1}) - {\xi}(p_{i_1})\right| 
    \vee \left| \hat{\xi}(p_{i_1+k}) - {\xi}(p_{i_1+k})\right| \leq
    \max_{p_{i_j} \in B_i}\left| \hat{\xi}(p_{i_j}) - {\xi}(p_{i_j})\right|.
\end{align*}
c) When $ c_i > {\xi}(p_{i_k}) \geq  \hat\xi(p_{i_k})$, we have
\begin{align*}
    \max_{p_{i_j} \in B_i}\left| c_i - {\xi}(p_{i_j})\right|  =
    \left| c_i - {\xi}(p_{i_1})\right| \leq \left| \hat{\xi}(p_{i_1}) - {\xi}(p_{i_1})\right|
    \leq \max_{p_{i_j} \in B_i}\left| \hat{\xi}(p_{i_j}) - {\xi}(p_{i_j})\right|.
\end{align*}
In case (iii), $\hat{\xi}(p_{i_k}) \geq \xi(p_{i_k})$ and the monotonicity of $\tilde{\xi}$ and ${\xi}$ imply that
\begin{equation*}
\max_{p_{i_j}\in B_i}\left| c_i - {\xi}(p_{i_j})\right| = \left| c_i - {\xi}(p_{i_1})\right| \leq
\left| \hat{\xi}(p_{i_1}) - {\xi}(p_{i_1})\right|
\leq \max_{p_{i_j}\in B_i} \left| \hat{\xi}(p_{i_j}) - {\xi}(p_{i_j})\right|.
    % {\xi}(p_{i_k}) \geq {\xi}(p_{i_1}) > \hat{\xi}(p_{i_1})
    % > c_i > {\xi}(p_{i_k}).
\end{equation*}
Recall that 
\begin{align*}
    \max_{p_j\in A}\left|\tilde{\xi}(p_j)-{\xi}(p_j)\right| =
    \max_{p_j\in A}\left|\hat{\xi}(p_j)-{\xi}(p_j)\right|.
\end{align*}
If $B=\emptyset$, then we have the equality in \eqref{e:sup_tilde_hat_inequality}, which ends the proof.
\end{proof}

%%%%%%%%%%%%%%%%%%%%%%%%%%%%%%%%%%%%%%%%%%%%%%%%%%%%%%%%%%%%%%%%%%%%%%%%%%%%%%%%%%%%%%%%%%%%%%%%%%%%%%%%%
\textbf{Proof of Theorem \ref{thm:IOQR_conv}. }
%%%%%%%%%%%%%%%%%%%%%%%%%%%%%%%%%%%%%%%%%%%%%%%%%%%%%%%%%%%%%%%%%%%%%%%%%%%%%%%%%%%%%%%%%%%%%%%%%%%%%%%%%

(i) We have to show that for every $i\in\left\{0,\ldots,d \right\}$
\begin{align*}
    \sup_{p \in [\eta,1-\eta]} \left| \hat{\beta}_i^{(m,n)}(p) - \beta_i(p) \right| 
    \xrightarrow[m,n\to\infty]{P}0
\end{align*}
which is equivalent to $\forall \varepsilon > 0$ 
\begin{align}\label{e:condP}
    \lim_{m,n\to\infty}
    P\left(
    \sup_{p \in [\eta,1-\eta]} \left| \hat{\beta}_i^{(m,n)}(p) - \beta_i(p) \right| \leq \varepsilon 
    \right) = 1.
\end{align}
Condition \eqref{e:condP} means that for every $\varepsilon>0$ and $\delta>0$ there exist $m_0\in\mathbb{N}$ and $n_0\in\mathbb{N}$ such that for every $m>m_0$ and $n>n_0$ we have
\begin{align*}
    P\left(
    \sup_{p \in [\eta,1-\eta]} \left| \hat{\beta}_i^{(m,n)}(p) - \beta_i(p) \right| \leq \varepsilon \right) \geq 1-\delta.
\end{align*}
Since $\beta_i$ is uniformly continuous on each interval $[\eta,1-\eta]$, for every $\varepsilon > 0$ there exists a~finite partition of $[\eta,1-\eta]$ defined by $m(\varepsilon)$ points, 
namely $\eta = p_1 < p_2 < \ldots < p_{m(\varepsilon)} = 1-\eta$, such that 
\begin{equation}\label{e:th3.1_1}
{\beta}_i(p_{j+1}) - \beta_i(p_j) < \frac{\varepsilon}{2}, 
\end{equation}
for all $j\in\{1,2,\ldots,m(\varepsilon)\}.$
For each point $p_j$, where $j\in\{1,2,\ldots,m(\varepsilon)\}$ the estimator $\hat{\beta}_i^{(m(\varepsilon),n)}(p_j)$, under some assumptions (see the remarks above Theorem \ref{thm:IOQR_conv}), 
is a~consistent estimator of ${\beta}_i(p_j).$
Thus, for all $\varepsilon>0$ and $\delta>0$ and for all $j$, there exists $n_j\in\mathbb{N}$, such that $\forall n>n_j$ with probability $1-\delta$, we have
$\left|\hat{\beta}_i^{(m(\varepsilon),n)}(p_j) - \beta_i(p_j) \right| < \frac{\varepsilon}{2}$. 
Let $n(\varepsilon) = \max_{j\in\{1,2,\ldots,m(\varepsilon)\}} \{ n_j \}$. 
Then $\forall n>n(\varepsilon)$, with probability $1-\delta$, we have
\begin{align}\label{e:th3.1_2}
    \left|\hat{\beta}_i^{(m(\varepsilon),n)}(p_j) - \beta_i(p_j) \right| < \frac{\varepsilon}{2}, \qquad \forall j\in\{1,2,\ldots,m(\varepsilon)\}.
\end{align}
Let $p\in[\eta,1-\eta]$. 
Therefore, $p\in[p_j,p_{j+1}]$ for some $j\in\{1,2,\ldots,m(\varepsilon)\}$, and one of two following cases holds 
\begin{equation}\label{e:th3.1a}
\hat{\beta}_i^{(m(\varepsilon),n)}(p_j) \leq \hat{\beta}_i^{(m(\varepsilon),n)}(p) \leq \hat{\beta}_i^{(m(\varepsilon),n)}(p_{j+1}); 
\end{equation}
\begin{equation}\label{e:th3.1b}
 \hat{\beta}_i^{(m(\varepsilon),n)}(p_j) \geq \hat{\beta}_i^{(m(\varepsilon),n)}(p) \geq \hat{\beta}_i^{(m(\varepsilon),n)}(p_{j+1}). 
\end{equation}
Taking into account \eqref{e:th3.1_1}, \eqref{e:th3.1_2} and the monotonicity of the function $\beta_{i}$
for every $n>n(\varepsilon)$, with probability $1-\delta$ we have 
\begin{align}\label{e:th3.1_3}
    \hat{\beta}_i^{(m(\varepsilon),n)}(p_j) 
    \geq {\beta}_i(p_j) - \frac{\varepsilon}{2}
    > \beta_i(p) - \varepsilon
\end{align}
and
% \begin{align}\label{e:th3.1_4}
%     \hat{\beta}_i^{(m(\varepsilon),n)}(p_{j+1}) 
%     \leq {\beta}_i(p_{j+1}) + \frac{\varepsilon}{2}
%     < \beta_i(p_j) + \varepsilon < \beta_i(p) + \varepsilon.
% \end{align}
\begin{align}\label{e:th3.1_4}
    \hat{\beta}_i^{(m(\varepsilon),n)}(p_{j+1}) 
    \leq {\beta}_i(p_{j+1}) + \frac{\varepsilon}{2}
    < \beta_i(p_j) + \varepsilon \leq \beta_i(p) + \varepsilon.
\end{align}
Inequalities \eqref{e:th3.1_3}, \eqref{e:th3.1_4} and inequality \eqref{e:th3.1a} which holds in one of the two possible cases, imply that for every $n>n(\varepsilon)$, with probability $1-\delta$
\begin{align*}
    \beta_i(p) - \varepsilon < \hat{\beta}_i^{(m(\varepsilon),n)}(p) < \beta_i(p) + \varepsilon.
\end{align*}
In the second case, when inequality \eqref{e:th3.1b} holds, we proceed analogously.
Namely, from inequalities \eqref{e:th3.1_1}, \eqref{e:th3.1_2} and the monotonicity of the function $\beta_{i}$
for every $n>n(\varepsilon)$, with probability $1-\delta$ we have
\begin{align*}
    \hat{\beta}_i^{(m(\varepsilon),n)}(p_j) 
    < {\beta}_i(p_j) + \frac{\varepsilon}{2}
    < \beta_i(p) + \varepsilon
\end{align*}
and
\begin{align*}
    \hat{\beta}_i^{(m(\varepsilon),n)}(p_{j+1}) 
    \geq {\beta}_i(p_{j+1}) - \frac{\varepsilon}{2}
    \geq \beta_i(p) - \frac{\varepsilon}{2}
    \geq \beta_i(p) - {\varepsilon}.
    % \hat{\beta}_i^{(m(\varepsilon),n)}(p_{j+1}) 
    % \geq {\beta}_i(p_{j+1}) - \frac{\varepsilon}{2}
    % \geq \beta_i(p_j) - \varepsilon.
\end{align*}
The above inequalities together with \eqref{e:th3.1b} give us that for every $n>n(\varepsilon)$, with probability $1-\delta$, 
\begin{align*}
    \beta_i(p) - \varepsilon < \hat{\beta}_i^{(m(\varepsilon), n)}(p) < \beta_i(p) + \varepsilon.
\end{align*}
Therefore, for every $m>m_0= m(\varepsilon)$ there exists $n_0= n(\varepsilon)$ such that for all $n>n_0$ we have
\begin{align*}
    P\left(
    \sup_{p \in [\eta,1-\eta]} \left| \hat{\beta}_i^{(m,n)}(p) - \beta_i(p) \right| \leq \varepsilon \right) \geq 1-\delta,
\end{align*}
which completes the proof of part (i) of the theorem.\\
%%%%%%%%%%%%%%%%%%%%%%%%%%%%%%%%%%%%%%%%%%%%%%%%%%%%%%%%%%%%%%%%%%%%%%%%%%%%%%%%%%%%%%%%%%%%%%%%%%%%%%%%%
(ii) For each $i\in\{0,1, \ldots, d\}$
\begin{align}\label{e:th3.1.ii1}
\sup_{p \in [\eta,1-\eta]} \left| \tilde{\beta}_i^{(m,n)}(p) - \beta_i(p) \right| \leq  \sup_{p \in [\eta,1-\eta]} \left| \tilde{\beta}_i^{(m,n)}(p) - \hat{\beta}_i^{(m,n)}(p)\right|
+ \sup_{p \in [\eta,1-\eta]} \left| \hat{\beta}_i^{(m,n)}(p) - \beta_i(p) \right|.
\end{align}
From (i), the second component in \eqref{e:th3.1.ii1} tends to 0 when $m$ and $n$ tend to infinity.
Thus, it is enough to show that
\begin{align*}
\sup_{p \in [\eta,1-\eta]} \left| \tilde{\beta}_i^{(m,n)}(p) - \hat{\beta}_i^{(m,n)}(p)\right| \xrightarrow[m,n\to\infty]{P}0 .
\end{align*}
Let $\eta=p_{1}<p_{2}<\ldots<p_{m-1}<p_{m}=1-\eta.$ Then,
\begin{equation*}
\sup_{p \in [\eta,1-\eta]} \left| \tilde{\beta}_i^{(m,n)}(p) - \hat{\beta}_i^{(m,n)}(p)\right| = \max_{j\in\{1,\ldots,m-1\}}\sup_{p \in [p_j,p_{j+1}]}\left| \tilde{\beta}_i^{(m,n)}(p) - \hat{\beta}_i^{(m,n)}(p)\right|.
\end{equation*}
From piecewise linearity of $\tilde{\beta}_i^{(m,n)}$ and $\hat{\beta}_i^{(m,n)}$
\begin{equation*}
\max_{j\in\{1,\ldots,m-1\}}\sup_{p \in [p_j,p_{j+1}]}\left| \tilde{\beta}_i^{(m,n)}(p) - \hat{\beta}_i^{(m,n)}(p)\right|=\max_{j\in\{1,\ldots,m\}}\left| \tilde{\beta}_i^{(m,n)}(p_{j}) - \hat{\beta}_i^{(m,n)}(p_{j})\right|.
\end{equation*}
Moreover,
\begin{eqnarray*}
\max_{j\in\{1,\ldots,m\}}\left| \tilde{\beta}_i^{(m,n)}(p_{j}) - \hat{\beta}_i^{(m,n)}(p_{j})\right| 
& \leq & \max_{j\in\{1,\ldots,m\}}\left| \tilde{\beta}_i^{(m,n)}(p_{j}) - \beta_i(p_{j})\right|\\ 
& & + \max_{j\in\{1,\ldots,m\}}\left| \hat{\beta}_i^{(m,n)}(p_{j}) - \beta_i(p_{j})\right|\\
&\leq & 2\max_{j\in\{1,\ldots,m\}}\left| \hat{\beta}_i^{(m,n)}(p_{j}) - \beta_i(p_{j})\right|,
\end{eqnarray*}
where the last inequality follows from Lemma \ref{lem:lower_sup}.
Finally,
\begin{eqnarray*}
\sup_{p \in [\eta,1-\eta]} \left| \tilde{\beta}_i^{(m,n)}(p) - \hat{\beta}_i^{(m,n)}(p)\right| & \leq & 2\max_{j\in\{1,\ldots,m\}}\left| \hat{\beta}_i^{(m,n)}(p_{j}) - \beta_i(p_{j})\right| \\
 &\leq & 2\sup_{p \in [\eta,1-\eta]} \left| \hat{\beta}_i^{(m,n)}(p) - \beta_i(p) \right|,
\end{eqnarray*}
and
\begin{eqnarray*}
\sup_{p \in [\eta,1-\eta]} \left| \tilde{\beta}_i^{(m,n)}(p) - \beta_i(p) \right| \leq 
3\sup_{p \in [\eta,1-\eta]} \left| \hat{\beta}_i^{(m,n)}(p) - \beta_i(p) \right|  \xrightarrow[m,n\to\infty]{P}0 
\end{eqnarray*}
from part (i) of the theorem and completes the proof of part (ii) of the theorem.\\
%%%%%%%%%%%%%%%%%%%%%%%%%%%%%%%%%%%%%%%%%%%%%%%%%%%%%%%%%%%%%%%%%%%%%%%%%%%%%%%%%%%%%%%%%%%%%%%%%%%%%%%%%
(iii) The estimator of $Q_x$ has the following form
 \begin{align*}
        \hat{Q}_{x}^{(m,n)}(p) = \psi^{-1}\left(\hat{\beta}_{0}^{(m,n)}(p)+\hat{\beta}_{1}^{(m,n)}(p)x_{1}+\ldots+ \hat{\beta}_{d}^{(m,n)}(p)x_{d}  \right),
    \end{align*}
    where $\psi$ is a~known continuous and strictly increasing function.
Therefore, convergence of $\hat{Q}_{x}^{(m,n)}$ to $Q_{x}$ follows from part (i) of the theorem, part (vi) of Theorem 18.10 in \cite{VanDerVaart1998}, and 
    the continuous mapping theorem (see, for example, \cite{VanDerVaart1998}, Theorem 18.11).
 Analogously, we can obtain the convergence of $\tilde{Q}_x^{(m,n)}$ to $Q_{x}.$\\
%%%%%%%%%%%%%%%%%%%%%%%%%%%%%%%%%%%%%%%%%%%%%%%%%%%%%%%%%%%%%%%%%%%%%%%%%%%%%%%%%%%%%%%%%%%%%%%%%%%%%%%%%
(iv) Convergence of the estimators $qZ(\cdot,\hat{Q}_{x}^{(m,n)})$ and $qZ(\cdot,\tilde{Q}_{x}^{(m,n)})$ follows from part (iii) of the theorem  and the continuous mapping theorem (see, for example, \cite{VanDerVaart1998}, Theorem 18.11). \\
%%%%%%%%%%%%%%%%%%%%%%%%%%%%%%%%%%%%%%%%%%%%%%%%%%%%%%%%%%%%%%%%%%%%%%%%%%%%%%%%%%%%%%%%%%%%%%%%%%%%%%%%%
(v) Convergence of $\int_{2\eta}^{1-2\eta}qZ\left(p;\hat{Q}_x^{(m,n)}\right)\,dp$ follows from the inequality
    \begin{eqnarray*}
    \left|\int_{2\eta}^{1-2\eta}\left[qZ\left(p;\hat{Q}_x^{(m,n)}\right)-qZ(p;Q_x)\right]\,dp \right|&\leq & \sup_{p\in(2\eta,1-2\eta)}\left|qZ\left(p;\hat{Q}_x^{(m,n)}\right)-qZ(p;Q_x)\right|,
    \end{eqnarray*}
    and part (iv) of the theorem. \\
Analogously, we can obtain the convergence of 
$\int_{2\eta}^{1-2\eta}qZ\left(p;\tilde{Q}_x^{(m,n)}\right)dp.$ 

%%%%%%%%%%%%%%%%%%%%%%%%%%%%%%%%%%%%%%%%%%%%%%%%%%%%%%%%%%%%%%%%%%%%%%%%%%%%%%%%%%%%%%%%%%%%%%%%%%%%
\subsection{Moments of the FLD and EFLD distributions}\label{sec:momentsEFLD}
%%%%%%%%%%%%%%%%%%%%%%%%%%%%%%%%%%%%%%%%%%%%%%%%%%%%%%%%%%%%%%%%%%%%%%%%%%%%%%%%%%%%%%%%%%%%%%%%%%%%

Sharma and Chakrabarty \cite{Sharma2019} showed that the mean and variance of $FLD(\alpha,\beta,\kappa)$ are given by
% \begin{equation*}
%     \mathbb{E} Y = \alpha + \frac{\beta \kappa}{2} \qquad \text{and} \qquad
%     \text{Var}(Y) = \beta^2 \left( \kappa + \frac{\kappa^2}{12} + \frac{\pi^2}{3} \right).
% \end{equation*}
\begin{equation*}
    \alpha + \frac{\beta \kappa}{2} \qquad \text{and} \qquad
    \beta^2 \left( \kappa + \frac{\kappa^2}{12} + \frac{\pi^2}{3} \right),
\end{equation*}
respectively. 
The following result provides the formula for the $r$-th moment of $EFLD(\alpha, \beta$ and $\kappa)$, which leads to the formulae for its mean and variance.
\begin{lemma}\label{lem:moments_efld}
    If $Y\sim \text{EFLD}(\alpha, \beta, \kappa)$ and $\beta r < 1$, then the $r$-th moment is given by
    \begin{equation*}
        m_r := \mathbb{E} (Y^r) = \exp(r\alpha)
        _1F_1(1+r\beta, 2; r\beta\kappa)
        \frac{\pi\beta r}{\sin(\pi\beta r)},
    \end{equation*}
where $_1F_1(a, c; z)$ is the confluent hypergeometric function of the first kind.
If $\beta r \geq 1$, then the $r$-th moment is not finite.
\end{lemma}
\begin{proof}
Using Taylor's expansion of the exponential function and changing the order of the integration and summation, it is easy to show that for $\beta r < 1$, the $r$-th moment is given by
    \begin{align}\label{eq:proof_efld_moments}
        m_r & = \mathbb{E} (Y^r) =
        \int_0^1 [Q_Y(p)]^r \, dp =
        \int_0^1 \exp(r\alpha) p^{r\beta}(1-p)^{-r\beta} \exp(r\beta\kappa p) \, dp \\
        \nonumber & = \exp(r\alpha) \Gamma(1-r\beta) \Gamma(1+r\beta) _1F_1(1+r\beta, 2; r\beta\kappa) \\
        \nonumber & = \exp(r\alpha) _1F_1(1+r\beta, 2; r\beta\kappa)
        \frac{\pi\beta r}{\sin(\pi\beta r)}.
    \end{align}
    When $\beta r \geq 1$, the integrand in the expression \eqref{eq:proof_efld_moments} is larger than $\exp({r\alpha})p^{r\beta}/(1-p)$ for each $p\in[0,1]$, hence the integral does not converge.
\end{proof}

\begin{corollary}
    If $Y\sim \text{EFLD}(\alpha, \beta, \kappa)$ and $\beta < 1$, then
    \begin{align*}
        \mathbb{E}(Y) = \exp(\alpha) {}_1F_1(1+\beta, 2; \beta\kappa) \frac{\pi\beta}{\sin(\pi\beta)}.
    \end{align*}
    For $\beta < \frac{1}{2}$,
    the variance $\Var(Y)$ is $m_2-m_1^2$, where $m_1$ and $m_2$ are defined as in Lemma \ref{lem:moments_efld}.
    For $\beta\geq1$ the mean of $Y$ is not finite and for $\beta \geq \frac{1}{2}$ the variance of $Y$ is not finite.
\end{corollary}
% \begin{corollary}
%     If $Y\sim \text{EFLD}(\alpha, \beta, \kappa)$ and $\beta < 1$, then
%     \begin{align*}
%         \mathbb{E}(Y) = \exp(\alpha) _1F_1(1+\beta, 2; \beta\kappa) \frac{\pi\beta}{\sin(\pi\beta)}
%     \end{align*}
%     and the variance $Var(Y)$ is $m_2-m_1^2$, where $m_1$ and $m_2$ are defined as in Lemma \ref{lem:moments_efld}.
% \end{corollary}

%%%%%%%%%%%%%%%%%%%%%%%%%%%%%%%%%%%%%%%%%%%%%%%%%%%%%%%%%%%%%%%%%%%%%%%%%%%%%%%%%%%%%%%%%%%%%%%%%%%%%%%%%
\subsection*{Data and Code Availability}\label{sec:data_and_code}
The data used in this study are publicly available from the \textit{census2000} dataset in the \textit{wooldridge} R package (Shea and Brown, 2021). 
The code used to reproduce the simulation study, the empirical analysis and the figures presented in this article is available on a~dedicated GitHub repository:
\url{https://github.com/skfp/conditional_measures}.
Additional simulation results and supplementary plots not included in the manuscript are also available in the repository.

%%%%%%%%%%%%%%%%%%%%%%%%%%%%%%%%%%%%%%%%%%%%%%%%%%%%%%%%%%%%%%%%%%%%%%%%%%%%%%%%%%%%%%%%%%%%%%%%%%%%%%%%%
\vspace{1cm}
\noindent
\subsection*{Acknowledgement.} S.P. and R.T. acknowledge support by the Dioscuri programme initiated by the Max Planck Society, jointly managed with the National Science Centre (Poland), and mutually funded by the Polish Ministry of Science and Higher Education and the German Federal Ministry of Research, Technology and Space.


\begin{thebibliography}{10}

\bibitem{Abrevaya2005}
{\sc Abrevaya, J.}
\newblock Isotonic quantile regression: Asymptotics and bootstrap.
\newblock {\em Sankhyā: The Indian Journal of Statistics (2003-2007) 67}, 2 (2005), 187--199.

\bibitem{Angrist2006}
{\sc Angrist, J., Chernozhukov, V., and Fernández-Val, I.}
\newblock Quantile regression under misspecification, with an application to the u.s. wage structure.
\newblock {\em Econometrica 74}, 2 (2006), 539--563.

\bibitem{Bagul2021}
{\sc Bagul, Y.~J., and Chesneau, C.}
\newblock Sigmoid functions for the smooth approximation to the absolute value function.
\newblock {\em Moroccan Journal of Pure and Applied Analysis 7}, 1 (2021), 12--19.

\bibitem{Barlow1972}
{\sc Barlow, R., Bartholomew, D.~J., Bremner, J.~M., and Brunk, H.~D.}
\newblock {\em Statistical Inference Under Order Restrictions: The Theory and Application of Isotonic Regression}.
\newblock J. Wiley, 1972.

\bibitem{Bassett1982}
{\sc Bassett, G., and Koenker, R.}
\newblock An empirical quantile function for linear models with iid errors.
\newblock {\em Journal of the American Statistical Association 77}, 378 (1982), 407--415.

\bibitem{Best1990}
{\sc Best, M., and Chakravarti, N.}
\newblock Active set algorithms for isotonic regression; a unifying framework.
\newblock {\em Math. Program. 47\/} (1990), 425--439.

\bibitem{Bondell2010}
{\sc Bondell, H.~D., Reich, B.~J., and Wang, H.}
\newblock {Noncrossing quantile regression curve estimation}.
\newblock {\em Biometrika 97}, 4 (2010), 825--838.

\bibitem{Brazauskas2024}
{\sc Brazauskas, V., Greselin, F., and Zitikis, R.}
\newblock Measuring income inequality via percentile relativities.
\newblock {\em Quality \& Quantity\/} (2024).

\bibitem{Costa2023}
{\sc Costa, M.}
\newblock Gender gap assessment and inequality decomposition.
\newblock In {\em Models for Data Analysis\/} (Cham, 2023), E.~Brentari, M.~Chiodi, and E.-J.~C. Wit, Eds., Springer International Publishing, pp.~109--124.

\bibitem{Farcomeni2024}
{\sc Farcomeni, A., and Geraci, M.}
\newblock Quantile ratio regression.
\newblock {\em Statistics and Computing 34\/} (2024).

\bibitem{Firpo2009}
{\sc Firpo, S., Fortin, N.~M., and Lemieux, T.}
\newblock Unconditional quantile regressions.
\newblock {\em Econometrica 77}, 3 (2009), 953--973.

\bibitem{Gini1912}
{\sc Gini, C.}
\newblock {\em Variabilit{\`a} e mutabilit{\`a}: contributo allo studio delle distribuzioni e delle relazioni statistiche. [Fasc. I.]}.
\newblock Studi economico-giuridici pubblicati per cura della facolt{\`a} di Giurisprudenza della R. Universit{\`a} di Cagliari. Tipogr. di P. Cuppini, 1912.

\bibitem{He2023}
{\sc He, X., Pan, X., Tan, K.~M., and Zhou, W.-X.}
\newblock Smoothed quantile regression with large-scale inference.
\newblock {\em Journal of Econometrics 232}, 2 (2023), 367--388.

\bibitem{He2003}
{\sc He, X., and Zhu, L.-X.}
\newblock A lack-of-fit test for quantile regression.
\newblock {\em Journal of the American Statistical Association 98}, 464 (2003), 1013--1022.

\bibitem{Horowitz1998}
{\sc Horowitz, J.~L.}
\newblock Bootstrap methods for median regression models.
\newblock {\em Econometrica 66}, 6 (1998), 1327--1351.

\bibitem{Hutson2024}
{\sc Hutson, A.~D.}
\newblock The generalized sigmoidal quantile function.
\newblock {\em Communications in Statistics - Simulation and Computation 53}, 2 (2024), 799--813.

\bibitem{Jedrzejczak2014}
{\sc Jędrzejczak, A.}
\newblock Income inequality and income stratification in {P}oland.
\newblock {\em Statistics in Transition new series 15}, 2 (2014), 269--282.

\bibitem{Koenker2005}
{\sc Koenker, R.}
\newblock {\em Quantile Regression}.
\newblock Cambridge University Press, 2005.

\bibitem{Koenker1978}
{\sc Koenker, R., and Bassett, G.}
\newblock Regression quantiles.
\newblock {\em Econometrica 46}, 1 (1978), 33--50.

\bibitem{Koenker2005a}
{\sc Koenker, R., and Ng, P.}
\newblock Inequality constrained quantile regression.
\newblock {\em Sankhyā: The Indian Journal of Statistics (2003-2007) 67}, 2 (2005), 418--440.

\bibitem{Larraz2014}
{\sc Larraz, B.}
\newblock Decomposing the {G}ini inequality index: An expanded solution with survey data applied to analyze gender income inequality.
\newblock {\em Sociological Methods \& Research 44}, 3 (2015), 508--533.

\bibitem{Lorenz1905}
{\sc Lorenz, M.~O.}
\newblock Methods of measuring the concentration of wealth.
\newblock {\em Publications of the American Statistical association 9}, 70 (1905), 209--219.

\bibitem{Moesching2020}
{\sc M{\"o}sching, A., and D{\"u}mbgen, L.}
\newblock {Monotone least squares and isotonic quantiles}.
\newblock {\em Electronic Journal of Statistics 14}, 1 (2020), 24--49.

\bibitem{Muggeo2012}
{\sc Muggeo, V.~M., Sciandra, M., and Augugliaro, L.}
\newblock Quantile regression via iterative least squares computations.
\newblock {\em Journal of Statistical Computation and Simulation 82}, 11 (2012), 1557--1569.

\bibitem{Prendergast2016}
{\sc Prendergast, L.~A., and Staudte, R.~G.}
\newblock {Quantile versions of the Lorenz curve}.
\newblock {\em Electronic Journal of Statistics 10\/} (2016), 1896--1926.

\bibitem{Prendergast2018}
{\sc Prendergast, L.~A., and Staudte, R.~G.}
\newblock A simple and effective inequality measure.
\newblock {\em The American Statistician 72}, 4 (2018), 328--343.

\bibitem{Saleh2021}
{\sc Saleh, R.~A., and Saleh, A. K. M.~E.}
\newblock Solution to the non-monotonicity and crossing problems in quantile regression.
\newblock \url{https://arxiv.org/abs/2111.04805}, 2021.

\bibitem{Sharma2019}
{\sc Sharma, D., and Chakrabarty, T.~K.}
\newblock The quantile-based flattened logistic distribution: Some properties and applications.
\newblock {\em Communications in Statistics - Theory and Methods 48}, 14 (2019), 3643--3662.

\bibitem{census2000}
{\sc {Shea, Justin M. and Brown, Kennth H.}}
\newblock census2000 dataset.
\newblock \url{https://rdrr.io/cran/wooldridge/man/census2000.html}, 2021.
\newblock Accessed: 2025-02-11.

\bibitem{VanDerVaart1998}
{\sc van~der Vaart, A.~W.}
\newblock {\em Asymptotic Statistics}.
\newblock Cambridge Series in Statistical and Probabilistic Mathematics. Cambridge University Press, Cambridge, 1998.

\bibitem{Wu2009}
{\sc Wu, Y., and Liu, Y.}
\newblock Stepwise multiple quantile regression estimation using non-crossing constraints.
\newblock {\em Statistics and Its Interface Volume 2\/} (2009), 299--310.

\end{thebibliography}
\end{document}